\documentclass{amsart}
\usepackage{amsfonts}
\usepackage{amssymb}
\usepackage{latexsym}
\usepackage{amscd} 
\usepackage{amsmath}
\usepackage{epsfig}
\usepackage{ifpdf}

\begin{document}
\title{Counting Hyperbolic Components}
\author{Jan Kiwi and Mary Rees}
\thanks{The idea for this paper first developed during a visit  to Liverpool by the first author in September 2006, partially funded by the LMS Scheme 4. The importance of this support is gratefully acknowledged. The first author is partially supported by the Research Network on Low Dimensional Dynamics ACT-17, Conicyt, Chile.}

\maketitle

\begin{abstract} We give formulae for the numbers of type II and type IV hyperbolic components in the space of quadratic rational maps, for all fixed periods of attractive cycles.
\end{abstract} 

\newtheorem{theorem}{Theorem}[section]
\newtheorem{introtheorem}{Theorem}
\newtheorem{introcorollary}[introtheorem]{Corollary}
\newtheorem{corollary}[theorem]{Corollary}
\newtheorem{conjecture}[theorem]{Conjecture}
\newtheorem{lemma}[theorem]{Lemma}
\newtheorem{definition}[theorem]{Definition}
\newtheorem{proposition}[theorem]{Proposition}
\newtheorem{remark}[theorem]{Remark}
\newtheorem{notation}[theorem]{Notation}

\numberwithin{equation}{section}

\newenvironment{ulemma}{\par\noindent\textbf{Lemma}\,\,\em}{\rm}
\newenvironment{utheorem}{\par\noindent\textbf{Theorem}\,\,\em}{\rm}
\newenvironment{uremark}{\par\noindent\textbf{Remark}\,\,}{\rm}
\newenvironment{udefinition}{\par\noindent\textbf{Definition}\,\,}{\rm}

\setcounter{equation}{0} 


\def\A{\mathbb{A}}
\def\B{\mathbb{B}}
\def\CC{\mathbb{C}}
\def\C{{\mathbb{C}}}
\def\F{\mathbb{F}}
\def\H{\mathbb{H}}
\def\R{\mathbb{R}}
\def\P{\mathbb{P}}
\def\Q{\mathbb{Q}}
\def\N{\mathbb{N}}
\def\L{\mathbb{L}}
\def\Z{\mathbb{Z}}
\def\D{\mathbb{D}}
\def\S{\mathbb{S}}
\def\T{\mathbb{T}}
\def\V{\mathbb{V}}

\def\QS{\Q / \Z}
\def\RZ{\R / \Z}

\def\Ctwo{\mathbb{C}^2}
\def\CPtwo{{\mathbb{CP}^2}}
\def\CPone{{\mathbb{CP}^1}}

\def\Nnot{\mathbb{N}_0}
\def\CDC{\CC \setminus \overline{\D}}

\def\ratdc{{\operatorname{Rat}^d_\C}}
\def\mtwocm{{{\cM}_2^{\mathrm{cm}}}}
\def\mtwofm{{{\cM}_2^{\mathrm{tm}}}}
\def\mtwofp{{{\cM}_2^{\mathrm{fm}}}}

\def\mindeg{\mathrm{mindeg}}
\def\mod{\mathrm{mod}}
\def\gcd{\operatorname{gcd}}
\def\cA{{\mathcal{A}}}
\def\cB{{\mathcal{B}}}
\def\cC{{\mathcal{C}}}
\def\cD{{\mathcal{D}}}
\def\cE{{\mathcal{E}}}
\def\cF{{\mathcal{F}}}
\def\cG{{\mathcal{G}}}
\def\cH{{\mathcal{H}}}
\def\cI{{\mathcal{I}}}
\def\cJ{{\mathcal{J}}}
\def\cK{{\mathcal{K}}}
\def\cL{{\mathcal{L}}}
\def\cM{{\mathcal{M}}}
\def\cN{{\mathcal{N}}}
\def\cO{{\mathcal{O}}}
\def\cP{{\mathcal{P}}}
\def\cQ{{\mathcal{Q}}}
\def\cR{{\mathcal{R}}}
\def\cS{{\mathcal{S}}}
\def\cT{{\mathcal{T}}}
\def\cU{{\mathcal{U}}}
\def\cV{{\mathcal{V}}}
\def\cW{{\mathcal{W}}}
\def\cX{{\mathcal{X}}}
\def\cY{{\mathcal{Y}}}
\def\cZ{{\mathcal{Z}}}

\def\ocP{{\overline{\cP}}}
\def\ocQ{{\overline{\cQ}}}
\def\ocD{{\overline{\cD}}}
\def\ocC{{\overline{\cC}}}

\setcounter{tocdepth}{2}


\section{Introduction}\label{1}

The aim of this paper is to solve some  basic counting problems which arise in the study of quadratic rational maps 
as dynamical systems acting on the Riemann Sphere.
In particular, given any  integers $n,m \ge 1$, we  compute the number of  quadratic rational maps 
such that both critical points are periodic, one of period $n$ and the other of period $m$.
A computation with a finite output is only possible (and interesting) if our counting takes place 
in an appropriate moduli
space, namely, the moduli space $\mtwocm$ formed by conjugacy classes of quadratic rational maps 
with marked critical points (see Section~\ref{ssec:statements} below).

Without neglecting the intrinsic interest that enumerative problems in
moduli spaces have, our motivation finds its origin in the study of
the open and conjecturally dense subset of $\mtwocm$ formed by
hyperbolic maps.  For short, we say that a connected component of this
subset is a hyperbolic component. Relevant dynamical features such as
number and period of attractors remain unchanged within a hyperbolic
component. In fact, the dynamics over the Julia set of maps within a
hyperbolic component is quasiconformally conjugate.  According
to~\cite{R4}, counting hyperbolic components and computing numbers as
the one mentioned above are essentially equivalent problems.

Hyperbolic maps are uniformly expanding on the Julia set and are
characterised as the maps for which all its critical points lie in the
basin of some attracting periodic orbit.  The various possible
combinatorial arrangements that the orbits of the Fatou components
containing critical points might have, give a first rough
classification of hyperbolic components.  More precisely, given a
hyperbolic component $\cH$ in $\mtwocm$, then one and exactly one of
the following holds for all maps in $\cH$~\cite{R4}:
\begin{itemize}
\item[(i)] Both critical points lie in the Fatou component of an
  attracting fixed point.

\item[(ii)] There is one periodic orbit of Fatou components and the
  critical points belong to different components of this orbit.

\item[(iii)] There is one periodic orbit of Fatou components and only
  one critical point belongs to a component of this orbit, and the other
  eventually maps into this orbit of components.

\item[(iv)] There are two periodic orbits of Fatou components, each one
  containing a critical point.
\end{itemize}
We say that $\cH$ is of type I,II, III or IV, according to which one of the statements above holds.

There is exactly one type I hyperbolic component in $\mtwocm$, it is  
formed by maps having a disconnected Julia set. The other hyperbolic components are formed by
maps having connected Julia set~\cite{R4}.

According to~\cite{R4}, each hyperbolic component $\cH \subset \mtwocm$ of
quadratic rational maps with connected Julia sets contains a unique critically marked
postcritically finite rational map, modulo conjugacy, called the
{\em centre} of the hyperbolic component. 
Thus, counting centres and counting hyperbolic components
is completely equivalent. 

A type II component $\cH$ such that its periodic orbit of Fatou components has period $n$ is centred at a map with both critical points in the same
cycle of period $n$. Similarly, a type IV component $\cH$, with the first critical point in a Fatou component of period $n$ and the second in one of period $m$, is centred at a map where the first critical point has period $n$ and the second $m$. In the first case we say that $\cH$ is a type II component of period $n$, and in the second case we say that $\cH$ is a type IV component of period $(n,m)$.

The aim of this paper is to obtain formulae for the following numbers:
\begin{eqnarray*}
  \eta_{\rm {II}} (n) & = & \# \{ \cH \subset \mtwocm \mid \cH \mbox{ is a type II component of period dividing } n \}, \\
\eta_{\rm{IV}} (n,m)  & = & \# \{ \cH \subset \mtwocm \mid \cH \mbox{ is a type IV component of period } (j,k) \mbox{ where } j | n \mbox{ and }  k | m \}.
\end{eqnarray*}

The value of $\eta_{\rm{IV}} (1,m)$ is well known.
In fact, since quadratic rational maps that fix a critical point are quadratic polynomials, modulo putting the fixed critical point at $\infty$. Thus,
 $\eta_{\rm{IV}}(1,m)$
is the number of elements of the quadratic family $Q_c(z) =z^2 +c$ for which the critical point $z=0$ is periodic of period dividing $m$. 
In the early 1980's, it was established that $\eta_{\rm{IV}}(1,m) = 2^{m-1}$,  (see~\cite[Expose XIX]{D-H1} for three different proofs of the fact that
all solutions of $Q_c^{m-1}(c)=0$ are simple).

For higher degree polynomials it is likely that the available techniques may lead to answers for analogue  counting problems. However, 
similar problems for moduli spaces of rational maps of degree $\ge3$ are open.

\subsection{Statement of the results}
\label{ssec:statements}
The elements of the {\it critically marked moduli space of quadratic rational maps} $\mtwocm$ are the conjugacy classes of triples $(f,\omega_1,\omega_2)$
where $f$ is a quadratic rational map with critical points at $\omega_1$ and $\omega_2$.
More precisely, $(f,\omega_1, \omega_2)$ and $(g, \omega'_1,\omega'_2)$ are conjugate if there exists a M\"obius transformation $\gamma$ such that $\gamma \circ f = g \circ \gamma$ and 
$\gamma (\omega_j) = \omega'_j$ for $j=1,2$. We will denote the conjugacy class of 
 $(f,\omega_1,\omega_2)$ by $[ f,\omega_1,\omega_2]$. 
The space  $\mtwocm$ is naturally identified to a complex algebraic surface with a unique
singular point at the centre of the unique type II component of period $2$, that is $[z^{-2}, 0 ,\infty]$ (see~\cite[Section~6]{M2}).



\smallskip
To state our results we will also need to introduce the numbers $\nu_q(n)$ as follows.
For $q >1$ and $n \ge 1$, let
$$ r \equiv n \, \mod  \, q,$$
such that $0 \leq r < q$. 

If $n < q$ define $\nu_q(n)=0$, otherwise let
$$\nu_q(n) = \begin{cases}
\dfrac{2^{n-1} - 2^{r-1}}{2^q -1} & \mbox{ if } q \nmid n, \\ \\
\dfrac{1}{2} + \dfrac{2^{n-1} - 2^{-1}}{2^q -1}  & \mbox{ if } q \mid n.
\end{cases}$$

In Lemma~\ref{lem:9}, we deduce from well known results about the quadratic polynomial family that
  $\nu_q(n)$ is the number of hyperbolic components of period dividing $n$ in a $p/q$-limb of the 
Mandelbrot set
(e.g. for the definition of limbs see~\cite{D-H1} or \cite{MilnorPOP}). 

Denote by $\phi(n)$ the Euler Phi function of $n$ (i.e., the number of integers $1 \leq k \leq n$ which are relatively prime to $n$).

We shall prove the following theorems.
\begin{theorem}\label{1.1}
For all  integers $m \geq n \geq 1$ we have:
\begin{equation}
\begin{split}
  \eta_{\rm IV} (n,m) = & \dfrac13 \left( 5 \cdot 2^{n+m-3}  + 2^{n-2} + 2^{m-2} \right)   - \dfrac{1}{2} \sum_{2 \leq q\leq n} \phi(q) \nu_q(n) \nu_q(m)
 \\ \\
                    &  - \eta_{\rm II}(\operatorname{gcd}(n,m)) + \dfrac16 ((-1)^n  + (-1)^m  + (-1)^{n+m}).\\                    
\end{split}
\end{equation}

In particular, given $n\geq 1$, 
$$\eta_{\rm IV} (n,m) =  \left( \dfrac{5}{3} \cdot 2^{n-3} + \dfrac{1}{12} - \dfrac14 \sum_{2 \leq q\leq n} \dfrac{\phi(q) \nu_q(n)}{2^q -1}  \right) \cdot 2^m + \varepsilon_n(m),$$
where $\varepsilon_n(m)$ is a bounded function of $m$. More precisely,
$$|\varepsilon_n(m)| \leq 2^n + 2^{2 \gcd(n,m)} .$$

\end{theorem}
For example, the asymptotic behaviour, as $m \to \infty$,
of the number of type IV hyperbolic components with one cycle of period exactly $n$ and the other of period dividing $m$, when $n \le 7$ is: 
$$\begin{array}{ll}
2^{m-1}&{\rm{\ if\ }}n=1,\\ \\
\dfrac{1}{3}\cdot 2^m+O(1)&{\rm{\ if\ }}n=2,\\ \\
\dfrac{23}{21}\cdot 2^{m}+O(1)&{\rm{\ if\ }}n=3,\\ \\ 
\dfrac{78}{35}\cdot 2^{m}+O(1)&{\rm{\ if\ }}n=4,\\ \\ 
\dfrac{6103}{1085}\cdot 2^{m}+O(1)&{\rm{\ if\ }}n=5,\\ \\
\dfrac{202371}{19530}\cdot 2^{m}+O(1)&{\rm{\ if\ }}n=6,\\ \\ 
\dfrac{29316701}{1240155}\cdot 2^m+O(1)&{\rm{\ if\ }}n=7.\end{array}$$

\begin{uremark}
  The analogue result for type III components is proved when $n=3$ in \cite{R1}, in a very simple-minded way, where it is pointed out that the type IV calculation can be done similarly. The simple-minded calculation for $n=3$ agrees with the above result.
\end{uremark}

\bigskip
It is easier to write the formula for $\eta_{\rm II}$  in terms of the number $\eta '_{\rm II}(m)$ 
of type II hyperbolic components of period exactly $m$, so that
$$\eta _{\rm II}(m)=\sum _{d\mid m}\eta '_{\rm II}(d).$$

\begin{theorem} \label{1.2} For $m\geq 3$,

\begin{equation}\label{1.2.1} \begin{array}{ll}{\displaystyle \sum _{d\mid m,d\geq 3}\dfrac{m}{d}\eta '_{\rm II}(d)}=&
\dfrac{7}{36}m2^m-\dfrac{37}{108}2^m-\dfrac{m}{4}-(-1)^m\dfrac{5}{36}m+\dfrac{1}{2}+(-1)^m\dfrac{5}{54}\\ \\
\ &{\displaystyle -\dfrac{1}{2}\sum _{3\leq q,q \le j \le m-q}\phi (q)\nu_q(j)\nu_q(m-j).}\end{array}\end{equation}
\end{theorem}

Thus, for $m \le 8$, the number  $\eta _{\rm II}'(m)$ 
of type II components of period $m$ is:
$$\begin{array}{ll}
1&{\rm{\ if\ }}m=2,\\
2&{\rm{\ if\ }}m=3,\\
6&{\rm{\ if\ }}m=4,\\
20&{\rm{\ if\ }}m=5,\\
46&{\rm{\ if\ }}m=6,\\
128&{\rm{\ if\ }}m=7,\\
284&{\rm{\ if\ }}m=8.
\end{array}$$
There is a (negative) contribution from $\phi (q)\nu _q(j)\nu _q(m-j)$ only for $m\geq 6$.

\medskip
As usual in enumerative problems of geometric 
nature, the proofs of our main results rely 
on counting a larger set contained in partial compactifications
of moduli space and then subtracting off the intersections at infinity. 
Mostly, we will consider a large portion $\cR$ of $\mtwocm$ and identify it with an open
and dense subset of
the projective plane $\CPtwo$.
The centres of the hyperbolic components of types II and IV of given periods will be the intersection locus in $\cR$ of two algebraic curves. To apply Bezout's Theorem, we use well known results which show that intersections are transverse and which are also useful to compute the degrees of these curves. These well known results are concerned with parametrisations of hyperbolic components of quadratic rational maps~\cite{R4}. 
To count  the intersections of these curves at infinity (i.e., outside $\cR$ in $\CPtwo$), we rely on results from Stimson's Thesis~\cite{Sti} as well as some new results.

\medskip
The proof of Theorem~\ref{1.1} regarding type IV components is contained in Section~\ref{2}, and the proof of Theorem~\ref{1.2} regarding type II components
is contained in Section~\ref{5}. In Section~\ref{sec:preliminaries} we 
discuss some of the necessary background results.

\section{Preliminaries: hyperbolic components and periodic curves}
\label{sec:preliminaries}
Recall that we work in the moduli space $\mtwocm$ of conjugacy classes
$[f, \omega_1, \omega_2]$ where $f$ is a quadratic rational map, and
$\omega_1, \omega_2$ are its critical points.  
According to~\cite{R4} each hyperbolic component of type II or IV in
$\mtwocm$ contains a unique postcritically finite quadratic rational
map, called the {\it centre} of the hyperbolic component.
We consider the {\it periodic curves} $V_n$ (resp. $W_m$) in $\mtwocm$
where $\omega_1$ has period exactly $n$ (resp. $\omega_2$ has period
exactly $m$).
Thus, our task
is related to understanding and computing the cardinality of $V_n \cap
W_m$, since the elements of this intersection are the centres of type
II or IV hyperbolic components with a Fatou component of period $n$
containing $\omega_1$ and a Fatou component of period $m$ containing
$\omega_2$.

\subsection{Parametrisation of type II and IV components}\label{type-param}
In order to compute the cardinality of $V_n \cap W_m$ it is convenient to know that $V_n$ and $W_m$ have transversal intersections except at the singular point of $\mtwocm$. This well known result is a direct consequence of the parametrisations~\cite{R4} of type II and IV hyperbolic components, which we proceed to describe.

According to part 2. (a) of the Main Theorem in~\cite{R4} we have that multipliers  parametrise type IV components:
\begin{theorem}
  \label{thr:IV}
Denote the open unit disk in $\C$ by $\D$.
  Let $\cH \subset \mtwocm$ be a type IV hyperbolic component. Given $\mathbf{f}= [f,\omega_1,\omega_2] \in \cH$, for $i=1,2$, denote by 
$\lambda_i (\mathbf{f})$ the multiplier of the attracting periodic orbit which contains $\omega_i$ in its basin.
Then the map $\lambda: \cH \to \D \times \D$ given by $\lambda(\mathbf{f})= (\lambda_1(\mathbf{f}),\lambda_2(\mathbf{f}))$ is 
biholomorphic.
\end{theorem}

Observe that a non-empty intersection of a periodic curve $V_n$ (resp. $W_m$) with a type IV hyperbolic component $\cH$ corresponds to the disk $\lambda_1 = 0$ (resp. $\lambda_2=0$).

\smallskip
In order to parametrise type II components we will model the first return map to critical Fatou components by a pair of Blaschke products of the form
$$\beta_a (z) = z \dfrac{1- \bar{a} }{1-a} \dfrac{z-a}{1- \bar{a} z},$$
where $a \in \D = \{ z \in \C \mid |z| <1 \}$.
Note that $\beta_a (\D) =  \D$, $\beta_a(0) =0$ and $\beta_a(1) =1$.

\begin{theorem}
  \label{thr:II}
Consider a period $m \ge 3$ type II hyperbolic component $\cH \subset \mtwocm$ such that the Fatou component containing the first critical point
maps in $j < m$ iterates onto the one containing the second critical point.
Then there exists a homeomorphism $h: \cH \to \D \times \D$
such that the following holds:

If $h(\mathbf{f}) = (a,b)$ where $\mathbf{f}=[f,\omega_1, \omega_2] \in \cH$  and 
$U_1$ (resp. $U_2$) is the Fatou component of $f$ containing $\omega_1$ (resp. $\omega_2$),
then for an appropriate choice of conformal maps $h_i : \D \to U_i$ where $i=1$ and $2$ we have that
$f^j : U_1 \to U_2$ coincides with $  h_2 \circ  \beta_a \circ h_1^{-1}$ and
$f^{m-j} : U_2 \to U_1$ coincides with $ h_1 \circ  \beta_b \circ h_2^{-1}$.
\end{theorem}

Although this theorem is a consequence of the main result in~\cite{R4}, 
it is easier to  deduce the above statement from the literature with the aid
of~\cite[Theorem~9.3]{M3}. 

\begin{proof}
In~\cite{M3}, Milnor works in the moduli space $\mtwofp$ of quadratic rational maps with marked fixed points. This moduli space $\mtwofp$ is formed by conjugacy classes of 
$(f, x_1, x_2, x_3)$
where $f$ is a quadratic rational map with fixed points at $x_1, x_2, x_3$ listed with repetitions according to multiplicity. The conjugacy class
of  $(f, x_1, x_2, x_3)$ is formed by the quadruples $(\gamma \circ f \circ \gamma^{-1}, \gamma(x_1),  \gamma(x_2), \gamma(x_3))$ where $\gamma$ is a M\"obius transformation. It follows that $\mtwofp$ is a complex (affine) algebraic surface
with a singularity at the class of the map possesing a triple fixed 
point~\cite[Lemma~6.6]{M2}. 
We will also employ the totally marked moduli space $\mtwofm$
defined as the conjugacy classes of $(f, \omega_1, \omega_2,  x_1, x_2, x_3)$ 
where $\omega_1, \omega_2$ are the critical points of $f$ and the $x_i$ are the fixed points as above. This latter moduli space is a smooth 
complex algebraic surface~\cite[Lemma~6.6]{M2}.
According to Milnor~\cite[page 51]{M2}, the forgetful map
$\mtwofm \to \mtwocm$ is a degree $2$ covering ramified only over $[z^{-2}, 0, \infty]$ and $\mtwofm \to \mtwofp$ is a degree $6$ covering ramified only over
the unique singular point of $\mtwofp$.

 Given a  type II hyperbolic component $\cH'$ in $\mtwofp$,   Theorem~9.3~\cite{M3}
produces a homeomorphism $h': \cH' \to \D \times \D$ that assigns to each element 
of $\cH'$ a pair of Blaschke products $(\beta_a, \beta_b)$  
which model the first return map to critical Fatou components (as in the statement of the theorem). 
Let us now translate this result to a period $m \ge 3$ type II hyperbolic component $\cH \subset \mtwocm$. 
From~\cite{R4} we know that $\cH$ is simply connected and
from~\cite[Theorem~9.3]{M3} the same holds for hyperbolic components in $\mtwofp$. Thus, via lifting $\cH$ to $\mtwofm$,
and then projecting into $\mtwofp$, we obtain a natural biholomorphic map from $\cH$ onto a hyperbolic component in $\mtwofp$.
The postcomposition of this homeomorphism with Milnor's parametrisation gives us the desired parametrisation $h: \cH \to \D \times \D$.
\end{proof}

\begin{corollary}
  \label{cor:TransversePeriodic}
For all $n \geq 1$ and $m \geq 1$,
 any intersection of $V_n$ with $W_m$ at a non-singular point of $\mtwocm$ is transverse.
\end{corollary}

\begin{proof}
Any point $\mathbf{f}$ in  $V_n \cap W_m$ lies in a type II or IV hyperbolic component. 
In view of theorems~\ref{thr:IV} and~\ref{thr:II}, there is a homeomorphism which maps
a neighbourhood of the origin in $\C^2$ onto a neighbourhood of $\mathbf{f}$ in $\mtwocm$
such that the germ of $\{ (x,y) \in \C^2 \mid xy=0 \}$ at the origin maps onto the germ of 
 $V_n \cup W_m$ at $\mathbf{f}$. From~\cite[Theorem~1]{Sa} it follows that
the intersection of $V_n \cup W_m$ with a small $3$-sphere around $\mathbf{f}$ 
consists of a link of two unknotted curves with linking number $1$.  Thus, the intersection at $\mathbf{f}$  is transverse.
\end{proof}


We will also need a result regarding transversal intersections of other critical orbit relations.
More precisely, given $m \ge 2$ and $1 \le j < m$ we consider the curve $\cP'_j$ consisting of all
elements $[f,\omega_1,\omega_2] \in \mtwocm$ such that $f^j(\omega_1) = \omega_2$. 
Also, we let $\cQ_{m-j}'$ be the curve formed by all $[f,\omega_1,\omega_2] \in \mtwocm$ such that $f^{m-j}(\omega_2) = \omega_1$.
Clearly, any intersection point of these curves is the centre of a type II hyperbolic component of period $n$ which divides $m$. 


\begin{corollary}
  \label{cor:TransversePreperiodic}
For all  $m \geq 3$ and $1 \le j < m$,
 any intersection of $\cP_j'$ with $\cQ_{m-j}'$ at a non-singular point of $\mtwocm$ is transverse.
\end{corollary}

\begin{proof}
To prove this corollary it is more convenient to reparametrise any type II hyperbolic component 
by pairs of Blaschke products having a critical point at the origin.
With this purpose, for each $a \in \D$ let us consider 
$$\gamma_a (z) =  \dfrac{1- \bar{a} }{1-a} \dfrac{z^2-a}{1- \bar{a} z^2}.$$
Note that $\gamma_a(\D) =\D$, $\gamma'(0) =0$, and $\gamma(1)=1$. 
Since $\gamma_a$ has only one critical point  $z =0$ in $\D$ and 
$\beta_a$ has only one fixed point $z=0$ in $\D$, it is not difficult to see
that there exists a diffeomorphism $\varphi : \D \times \D \to \D \times \D$
such that if $(c,d) = \varphi(a,b)$, then there exist automorphisms $h_1,h_2$ of $\D$ such that $\gamma_c = h_2^{-1} \circ \beta_a \circ h_1$ and 
$\gamma_d = h_1^{-1} \circ \beta_b \circ h_2$.  

It follows that given a period $n \ge 3$ type II hyperbolic component $\cH$, a non-empty intersection $\cP_j' \cap \cH$ 
is mapped under $\varphi \circ h: \cH \to \D \times \D$ onto $\{ 0 \} \times \D$ where $h$ is as in Theorem~\ref{thr:II} and $\varphi$ is as in the previous paragraph. Similarly $\cQ_{m-j}' \cap \cH$ is mapped onto $\D \times \{ 0 \} $. The corollary now follows from~\cite[Theorem~1]{Sa} as in the previous proof.
\end{proof}

\subsection{Periodic curves: smoothness}\label{2.1} 
Although periodic curves are well known to be smooth, we are unable to provide a published reference for this fact. A sketch of its proof 
is included below only for  the sake of completeness, since we make no essential use of it.

\begin{theorem}
  \label{thr:smooth}
  For all $n \geq 1$ and $m \geq 1$, the curves $V_n$ and $W_m$ are smooth at non-singular points of $\mtwocm$.
\end{theorem}

\begin{proof}{(Sketch)}
  For $n,m \leq 2$ this can be checked by a direct calculation, so we assume that $n$ and $m$ are at least 
$3$. To fix ideas we prove smoothness for $V_n$, smoothness of $W_m$ follows along the same lines.
We consider $\mathbf{g} = [g, \omega_1(g), \omega_2(g)] \in V_n$ and proceed to prove that $V_n$ is
smooth in a neighbourhood $U$ of $\mathbf{g}$.

First assume that $\omega_2(g)$ is not in the inmediate basin of the period $n$ cycle containing $\omega_1(g)$.
Normalising the critical points to $0$ and $\infty$, and one of the critical values to $1$, 
it is not difficult to conclude that there exists a holomorphic section $(f_u, \omega_1(u), \omega_2(u))$ 
of quadratic rational maps 
defined for all $u$ neighbourhood $U$ of $\mathbf{g}$ (i.e. $u = [ f_u, \omega_1(u), \omega_2(u)]$ for all $u \in U$).
Taking $U$ sufficiently small we may assume that the periodic orbit of $\omega_1(u)$ for $u = \mathbf{g}$
has a well defined  analytic continuation to a periodic orbit $\cO_u$ of $f_u$ for all $u \in U$. 
Let $\lambda: U \to \C$ be the map that assigns to each $u \in U$ the multiplier $\lambda(u)$ of $\cO_u$.
Note that $\lambda (u) =0$ if and only if $u \in V_n$. 
We claim that the gradient of $\lambda$ does not vanish at $\mathbf{g}$. 
In fact, note that each component of the inmediate basin of $\cO_{\mathbf{g}}$ is simply connected. Thus, 
we may apply quasiconformal surgery, as in the 
quadratic polynomial case (see~\cite[Th\'eor\'eme 4]{D}) 
to obtain a smooth family $u(\mu) \in U$ such that
$\lambda(u(\mu)) = \mu$, defined for a parameter $\mu$ varying in a neighbourhood of $0 \in \D$ and such that
$u(0) = \mathbf{g}$. Hence, the gradient of $\lambda$ does not vanish at $\mathbf{g}$, which guarantees smoothness
of $V_n$ near $\mathbf{g}$.

\medskip
\begin{uremark}
  The original surgery described in~\cite[Th\'eor\'eme 4]{D} yields a continuous family $u(\mu)$. However, following the proof of Theorem~5.8 in 
\cite{M3} the surgery may be  upgraded to one producing a real analytic family $u(\mu)$.
\end{uremark}

\medskip
It remains to check that $V_n$ is smooth at elements where the second critical point is in the basin of the first periodic critical point. This again follows from Theorem~\ref{thr:II}, since such elements of $V_n$ are contained in
type II hyperbolic components and, after applying~\cite[Theorem~1]{Sa}, we conclude that $V_n$ intersects a small $3$-sphere around 
these points in an unknotted simple closed curve. Thus, $V_n$ is also smooth at these points.
\end{proof}

\subsection{A convenient subset $\cR$ of $\mtwocm$ and projective curves}
\label{ss:curves}
For reasons that will be apparent later, for $n, m \geq 1$, we  consider 
the curve $X_n \subset \mtwocm$ consisting on all maps such that the first critical point is periodic of period
dividing $n$, and similarly the curve $Y_m$ formed by maps with the second critical point periodic of period dividing $m$.
The curves $X_n$ and $Y_m$ are the union of periodic curves.
That is,
\begin{eqnarray*}
  X_n &=& \cup_{p|n} V_p, \\
  Y_n &=& \cup_{p|n} W_p. 
\end{eqnarray*}

  For our purpose it is convenient to work with the set $$\cR = \mtwocm \setminus X_2$$ which
may be parametrised as follows. For  $(c,d) \in \C \times \C^*$, consider the quadratic rational map
$$f_{c,d}(z) = 1 + \dfrac{c}{z} + \dfrac{d}{z^2}.$$
Then $(c,d) \mapsto [f_{c,d}, 0, -2d/c]$ parametrises  $\cR= \mtwocm \setminus X_2$.
That is we identify $\cR$ with $\C \times \C^*$.

For us, it is also convenient to regard $\cR$ as a subset of  $\CPtwo$.  
By adding the line $d=0$ and a projective line at infinity to $\cR$, 
we obtain $\CPtwo$, with a preferred affine plane $\C^2$ parametrised by $(c,d)$.
Thus, we will regard $\cR$ both as a subset of $\mtwocm$ and of $\CPtwo$ according to convenience.

In order to be precise we let $\cX_n = X_n \cap \cR \subset \CPtwo$
and $\cY_m= Y_m \cap \cR \subset \CPtwo$. Similarly, let
$\cV_n = V_n \cap \cR$ and $\cW_n = W_n \cap \cR$. Denote by $\overline{\cX_n}$,
$\overline{\cY_m}$, $\overline{\cV_n}$ and, $\overline{\cW_n}$
 their closure in $\CPtwo$.
It follows that  $\overline{\cX_n}$,
$\overline{\cY_m}$, $\overline{\cV_n}$ and, $\overline{\cW_n}$ are projective algebraic varieties.

\section{Type IV components}\label{2}
The general strategy to compute $\eta_{\rm IV} (n,m)$ involves two main steps.
The first one consists of computing the degrees of the curves $\overline{\cX_n}$ and $\overline{\cY_m}$ introduced above, in Section~\ref{ss:curves}. 
Then we apply Bezout's Theorem to obtain the total number of intersections of these curves. Since we are only interested in intersections which are relevant to our count (that is, those in $\cR$), the second main step is to compute the number of intersections at ``infinity'' and subtract them from the total number of intersections. Both steps involve translating dynamical information into algebraic information about intersections. While in the first step, the main ingredients come from results summarised in Section~\ref{sec:preliminaries} about transversal intersections, in the second step, Stimson's results (Theorem~\ref{3.4}) about how certain dynamical systems are organised close to infinity in parameter space, will play a key role.

Complementary to the main ideas described in the previous paragraph,  
the proof of  Theorem~\ref{1.1} involves rewriting the formula as follows:
\begin{equation}\label{1.1.1}
\begin{split}
  \eta_{\rm IV} (n,m) =& \dfrac{1}{36}(2^n-3-(-1)^n))(7 \cdot 2^m+3-(-1)^m)\\ \\
                   & - \dfrac{1}{2} \sum_{3\leq q\leq n} \phi(q) \nu_q(n) \nu_q(m) \\ \\ 
                   & - \dfrac12 \left(\dfrac{2^n}{6}+ \dfrac{(-1)^n}{3}\right)  \left(\dfrac{2^m}{6}+ \dfrac{(-1)^m}{3}\right) \\ \\
                   & + \dfrac12  \left(\dfrac{2^n}{6}+ \dfrac{(-1)^n}{3}\right)  \left(\dfrac{2^m}{6}+ \dfrac{(-1)^m}{3}\right) \\ \\ 
                   & +\dfrac{ (4+(-1)^n)}{6} 2^m - \dfrac{(1+(-1)^n)(-1)^m }{6}   +\dfrac{(1+(-1)^n)(1+(-1)^m)}{4} \\ \\
                   & - \eta_{\rm II}(\gcd (n,m)).
\end{split}
\end{equation}
The first two lines add up to the size of $\cX_n \cap \cY_m$, for $n \geq 3, m \geq 1$. That is, 
the total number of type II and IV components with an attracting cycle of period $\geq 3$ dividing $n$ 
and an attracting cycle of period dividing $m$ (maybe the same cycle). 
The first line is the product of the degrees of the curves $\overline{\cX_n}$ and $\overline{\cY_m}$ (see Subsection~\ref{2.2}) and
the second line is their intersection number outside $\cR$ (see Sections~\ref{3.1} and~\ref{3.2}). 

The third and fourth line cancel out. However, the third line is $- (1/2) \phi(2) \nu_2(n) \nu_2 (m)$ so that we may insert this number in the 
sum of the second line of (\ref{1.1.1}).

The fifth line corresponds to $\eta_{\rm IV}(1,m)$ if $n$ is odd, to $\eta_{\rm IV}(2,m)$ if $n$ is even and $m$ is odd, and to  $\eta_{\rm IV}(2,m) + \eta_{\rm II}(2)$
if $n$ and $m$ are even.

Thus, the sum up to the fifth line gives the number of hyperbolic components with one critical Fatou component of 
period dividing $n$ and another of period dividing $m$. 
This count includes some type II components, namely, the ones with period dividing both $n$ and $m$.
The last line is the necessary correction to not count these $\eta_{\rm II}(\gcd (n,m))$ components
and only consider type IV components.

\subsection{The degrees of the curves}\label{2.2}

\begin{lemma} 
\label{lem:6}
The following statements hold:
\begin{itemize}
\item For all $n \ge 3$, the degree of $\overline{\cX_n}$ is
$$ \frac{1}{6}2^{n}-\frac{1}{2}-\frac{1}{6}(-1)^{n}.$$
\item For all $m \ge 1$,  the degree of $\overline{\cY_m}$ is
$$\frac{7}{6}2^{m}+\frac{1}{2}- \frac{1}{6}(-1)^{m}.$$
\item None of the points $[1:0:0], [0:1:0], [0:0:1]$ belongs to $\overline{\cX_n}$, for all $n \geq 3$.
\end{itemize}
\end{lemma}

Before proving the lemma, it is convenient to introduce notation, and state and prove some intermediate results
contained in Lemma~\ref{lem:4} below. These results will also be useful when counting type II components.

To compute the degree of the curves $\overline{\cX_n}$ we recursively define, 
for $n \ge 3$, polynomials $P_n(c,d),  Q_n(c,d) \in \C[c,d]$ as follows:
$$P_{3}(c,d)=1+c+d,\ \ Q_{3}(c,d)=1,$$
$$P_{n+1}=P_{n}^{2}+cP_{n}Q_{n}+dQ_{n}^{2},\ \ Q_{n+1}=P_{n}^{2}.$$
With these definitions we have that, for all $(c,d) \in \cR$, 
$$f_{c,d}^{n}(0)=\dfrac{P_{n}(c,d)}{Q_{n}(c,d)}.$$
Thus $\cX_n$ is the set of all $(c,d) \in \cR$ such that
$P_n(c,d)=0$.

Similarly, to compute the degrees of the curves $\overline{\cY_m}$ we define, for $m \ge 1$, 
$$R_{1}(c,d)=4d-c^{2},\ \  S_{1}(c,d)=4 d,$$
$$R_{m+1}=R_{m}^{2}+cR_{m}S_{m}+dS_{m}^{2},\ \ S_{m+1}=R_{m}^{2}.$$
Note that
$$f_{c,d}^{m}(-2d/c)=\frac{R_{m}(c,d)}{S_{m}(c,d)}.$$
Thus,  $\cY_m$ is the set of all $(c,d) \in \cR$ such that
 $c R_m(c,d) + 2d S_m (c,d) =0$.

 \begin{lemma}
   \label{lem:4}
   For all $n \ge 3$ and $m \ge 1$ the following statements hold:
   \begin{enumerate}
   \item $$\deg (P_n) = \frac{1}{6}2^{n}-\frac{1}{2}-\frac{1}{6}(-1)^{n},$$
$$\deg (Q_n) = \frac{1}{6}2^{n}-{1}+\frac{1}{3}(-1)^{n}.$$
   \item For all even  $n$, $\deg (P_n) = \deg (Q_n)$. For all odd $n$, $\deg (P_n) = \deg (Q_n) +1$.
   \item The constant term of $P_n$ and of  $Q_n$ is $1$.
   \item Among the terms of $P_n$ (resp. $Q_n$) with maximal degree $\delta=\deg (P_n)$ (resp. $\delta=\deg(Q_n)$), both the
     monomial $c^{\delta}$ and the monomial $d^{\delta}$ always appear multiplied by a positive integer coefficient.
   \item The minimum degree of the monomials of $R_m$ and of $S_m$ is $2^{m -1}$ and are uniquely realised by monomials in $d$. 
   \item Let $G_m (c,d) = c R_m(c,d) + 2d S_m (c,d)$. Then
     $$ \deg (G_m) = \frac{7}{6}2^{m}+\frac{1}{2}-\frac{1}{6}(-1)^{m}.$$
\end{enumerate}
 \end{lemma}

 \begin{proof}
   Since (i) through (iv) hold for $n=3$, we proceed by induction assuming that these assertions are true for $n$. 
From the definitions,  $P_{n+1} (0,0) = P^2_n (0,0) = Q_{n+1}(0,0)$, thus (iii) holds for $n+1$, since $P_n(0,0) =1$.
When $n$ is even,  from (iii), (iv)  and the formulae $P_{n+1} = P_n^2 +cP_{n}Q_{n}+dQ_{n}^{2}$, $Q_{n+1} = P_n^2$, it follows that 
(i), (ii) and (iv) hold for $n+1$ since $\deg (P_{n+1}) = 1 + \deg (P_n) + \deg (Q_n) = 1 + 2 \deg (P_n) = 1 + \deg (Q_{n+1})$.
Similarly when $n$ is odd, (i), (ii) and (iv) also hold for $n+1$ since $\deg (P_{n+1}) = 1 + \deg (P_n) + \deg (Q_n) =  2 \deg (P_n) = \deg (Q_{n+1})$.

Assertion (v) clearly holds for $m=1$. Assuming that (v) holds for $m$, from  the definitions it follows that 
the monomials of minimum degree of $R_{m+1}$ and of $S_{m+1}$ coincide with the square of the monomial of minimum degree of $R_m$.  
Thus, (v) holds for $m+1$ and the assertion holds, by induction.

To prove (vi),  we first establish by induction the assertions that $\deg(R_m) \ge \deg(S_m)$, and for all $m \ge 2$, 
among the terms of $R_m$ (resp. $S_m$) with maximal degree $\delta=\deg (R_m)$ (resp. $\delta=\deg(S_m)$) 
the monomial $c^{\delta}$ appears multiplied by a positive integer coefficient. In fact, for $m =2$, directly from the
definitions it follows that $c^4$ is a term of maximal degree of $R_2$ and $S_2$. Assuming the assertions true for $m$,
it follows that $S_{m+1} = R_m^2$ has degree $2\deg (R_m)$ and a term $c^{\deg (S_{m+1})}$ multiplied by a positive integer.
Moreover, $\deg (R_{m+1}) = \max \{ 2\deg (R_m), 1 + \deg (R_m) + \deg (S_m)\}$ and in all cases $R_{m+1}$ has a term of maximal degree
of the form $c^\delta$ multiplied by a positive integer coefficient. We also conclude that $\deg(R_{m+1}) \ge \deg(S_{m+1})$.

From the previous paragraph, it is not difficult to conclude that
for all  $m \ge 1$ odd, $\deg (R_{m+1}) = \deg (S_{m+1}) = 2 \deg (R_m)$ and, 
for all $m \ge 1$ even, $\deg (R_{m+1}) = \deg (S_{m+1}) +1 = 2 \deg (R_m) +1$.
Also, $ \deg (G_m) = \deg(R_m) +1$, since for $\delta=\deg (R_m)$ 
the monomial $c^{\delta}$ appears multiplied by a non-zero coefficient as a term of  $R_m$.
Now an induction readily checks that $ \deg (G_m)$ is given by the formula of assertion (vi).
\end{proof}

\begin{proof}[of Lemma~\ref{lem:6}]
Denote by $P^h_n (c,d,e)$ (resp. $G^h_m (c,d,e)$) the homogeneous version
of $P_n(c,d)$ (resp. $G_m(c,d)$). 
Then $\overline{\cX_n}$ (resp. $\overline{\cY_m}$) is the projective curve 
where  $P^h_n (c,d,e)$ (resp. $G^h_m (c,d,e)$) vanishes.

From Lemma~\ref{lem:4} (iii), we conclude that $[0:0:1] \notin \overline{\cX_n}$ since $P^h_n (0,0,1)=1$.
Lemma~\ref{lem:4} (iv)  implies that $P^h_n (1,0,0) \neq 0 \neq P^h_n (0,1,0)$.
Thus, none of the points $[1:0:0], [0:1:0], [0:0:1]$ belongs to $\overline{\cX_n}$.

Since $\deg (P^h_n) = \deg (P_n)$ and 
$ \deg (G^h_m) = \deg(G_m)$, in view of Lemma~\ref{lem:4},  it is sufficient to establish that $P^h_n$
generates the ideal of $\overline{\cX_n}$ and, similarly that $G^h_m$
generates the ideal of $\overline{\cY_m}$.

The birational automorphism of $\CPtwo$ induced by interchanging the role of the critical points
will allow to us to only check the above for $P^h_n$. More precisely, for all 
$(c,d) \in \cR \setminus \cY_2$, let $M_{c,d}$ be the M\"obius transformation such that:
$$ M_{c,d} (f^j_{c,d}(-2d/c)) =
\begin{cases}
0 & \mbox{ if } j=0,\\
\infty & \mbox{ if } j=1, \\
1 & \mbox{ if } j =2.
\end{cases}$$
Then there exists $(c',d') \in \cR$ such that
$$f_{c',d'} =  M^{-1}_{c,d} \circ f_{c,d} \circ  M_{c,d}.$$
It follows that $c'$ and $d'$ are rational functions of $(c,d)$.
Moreover, $M_{c',d'} =  M^{-1}_{c,d}$. Therefore, the map
$(c,d) \mapsto (c',d')$ is a birational map $\varphi : \CPtwo \dashrightarrow \CPtwo$.
Furthermore, 
$$\frac{G_m^h}{G} = \varphi^* (P^h_m)= P^h_m \circ \varphi,$$
where $G=G_1^h$ if $m$ is odd, and $G=G_2^h$ if $m$ is even.
After checking that $G_1^h$ and  $G_2^h$ are free of perfect square factors in $\C[c,d,e]$,
it follows that $G_m^h$ generates the ideal of  $\overline{\cY_m}$ if and only if $P^h_m$
generates the ideal of $\overline{\cX_m}$ for all $m \geq 3$.

To show that  $P^h_n$ is square factor free, it is sufficient to show that the degree of $\overline{\cX_n}$
coincides with that of $P_n$. For this we count the intersections of $\overline{\cX_n}$
with the degree $3$ curve $\overline{\cY_1}$. The curve $\overline{\cY_1}$ is defined by the equation
$c^3 - 4dce - 8 d^2 e=0$. Thus, $\overline{\cY_1}$ intersects the line $d=0$ at $[0:0:1]$ and the line $e=0$ at $[0:1:0]$.
Therefore  $\overline{\cY_1} \cap \overline{\cX_n}$ is contained in $\cR$. 
Hence, $\overline{\cY_1} \cap \overline{\cX_n}$ consists of points of transverse intersection between $ \cY_1$  and $\cX_n$.
It follows that $3 \deg(\overline{\cX_n})$ coincides with the cardinality of $ \cY_1 \cap \cX_n$.
Moreover, $\{ f_{c,d} \mid (c,d) \in \cY_1 \}$ is, modulo change of coordinates, the quadratic family $\{ z^2 +v \mid v \in \C \}$.
Thus, points in ${\cY_1} \cap {\cX_n}$ are in one to one correspondence with polynomials of the form $z^2 +v$
such that $z=0$ is periodic of period $q$ where $3 \leq q \mid n$. More precisely, the cardinality of
${\cY_1} \cap {\cX_n} = 2^{n-1} -1 -\delta$ where $\delta = 1$ if $n$ is odd, and $\delta = 2$ if $n$ is even.
That is,
$$3 \deg(\overline{\cX_n}) = 2^{n-1} -1 - \dfrac{1 + (-1)^n}{2} = 3 \deg (P^h_n),$$
and the lemma follows.
\end{proof}

\begin{uremark}
  From the proof above it follows that the birational map $\varphi: \CPtwo \dashrightarrow \CPtwo$ has degree~$7$.
\end{uremark}

\bigskip
To prove Theorem~\ref{1.1},  we need to compute the cardinality of the intersection
$$\cX_{n}\cap \cY_{m}.$$
But by Bezout's theorem, the number of intersections (with multiplicities) of 
$\overline{\cX_{n}}$ and $\overline{\cY_{m}}$
in $\mathbb C\mathbb P^{2}$ is simply the product of their degrees. 
So now we need to compute the number of intersections, with multiplicities, in $\mathbb C\mathbb P^{2}\setminus {\mathcal{R}}$.
Counting the intersections in  $\mathbb C\mathbb P^{2}\setminus {\mathcal{R}}$ is divided into two parts. First we show that there are no 
intersections in the line $[c:d:0]$, and then we count the intersections in the line $[c:0:1]$.

\subsection{No intersections at $c=\infty$ or $d=\infty $.}\label{3.1}

\begin{lemma}
\label{lem:7}
For all $n \geq 3$ and $m \geq 1$,
$$\overline{\cX_n} \cap \overline{\cY_m} \cap \{ x \in \CPtwo \mid x = [c:d:0] \} = \emptyset.$$
\end{lemma}

\begin{proof}
  Since $[1:0:0]$ and  $[0:1:0] \notin \overline{\cX_n}$ it is sufficient to consider  $x = [1:s:0]$ with $s \neq 0$ 
and show that $x \notin \overline{\cX_n} \cap \overline{\cY_m}$.
First note that if $f_{c,d} \in \cR$, then 
$$f^2_{c,d}(z) = 1 + \dfrac{z^2}{\dfrac{z^2}{c} + z + \dfrac{d}{c}} +  \dfrac{d}{c^2} \dfrac{z^4}{\left( \dfrac{z^2}{c} + z + \dfrac{d}{c} \right)^2}.$$
Thus as $[c:d:1]$ converges to $x$,
$$f^2_{c,d}(z) \to 1 +  \dfrac{z^2}{z + s} = g_s (z)$$
uniformly in compact subsets of $\C \setminus \{ -s \}$.

Note that $g_s$ has a parabolic fixed point at $\infty$, and that $g_s(-s) = \infty$.
Moreover, since one of the forward orbits of the two critical points $0, -2s$ of $g_s$ must be infinite and converge to $\infty$,
 one of the two critical points of $g_s$, call it $\omega$,  has an infinite forward orbit entirely contained
in  $\C \setminus \{ -s \}$. 

If $[c:d:1]$ is sufficiently close to $x$, then the first $\max\{n,m\}+1$ iterates of $\omega_1=0$ (when $\omega =0$) or of $\omega_2 = -2d/c$ (when $\omega=-2s$) are pairwise distinct. Thus, for all $[c:d:1]$ sufficiently close to $x$, we have that  $[c:d:1] \notin {\cX_n} \cup {\cY_m}$. 
Therefore, $x \notin \overline{\cX_n} \cap \overline{\cY_m}$.
\end{proof}

\subsection{Counting intersections at $d=0$.}\label{3.2}
Now our aim is to count the intersections at $d=0$. We start by establishing where the intersections occur.

\begin{lemma}
\label{lem:8}
For all $n \geq 3$, if  
$[c:0:1] \in \overline{\cX_n}$ then 
$c^{-1} = - {4 \cos^{2} \pi p/q}$ for some $1 \le p < q \le n $ with $\gcd(p,q)=1$ and $q \neq 2$.
\end{lemma}

\begin{proof}
For all $(c,d) \in \cR$, we have that the cross ratio 
$$[\omega_1, \omega_2, f_{c,d}(\omega_1), f_{c,d} (\omega_2)]= \dfrac{c^3}{8d^2} - \dfrac{c}{2d},$$
where $\omega_1 = 0, \omega_2 = -2d/c$ and
$$[0,z_2,\infty,z_4]=\frac{z_4}{z_2}.$$
It follows that, given $c_0 \neq 0$ and a sequence $(c_k,d_k) \in \cR \cap \cX_n$ which converges to $(c_0,0)$, the conjugacy
class of $f_{c_k,d_k}$ diverges to infinity in moduli space. According to~\cite[Lemma~4.1]{M2}),  $f_{c_k,d_k}$ has
three fixed points, one with multiplier diverging to $\infty$ and the other two multipliers converge
to  reciprocal roots of unity of order $q$ where $2 \leq q \leq n$.
Uniformly on compact subsets of $\C^*=\C \setminus \{ 0 \}$, the maps  $f_{c_k,d_k}$ converge to 
$$M_{c_0} (z) = 1 + \dfrac{c_0}{z}.$$
Since $\C^*$ contains the fixed points of $M_{c_0}$, the multipliers of the fixed points of
$M_{c_0}$ are of the form $\exp(\pm 2 \pi i p/q)$ where $\gcd(p,q)=1$. 
It follows that
$$c^{-1}_0 = -4 \cos^2(\pi p/q).$$
In particular, $q \neq 2$ and the lemma follows. 
\end{proof}

The rest of this section is devoted to proving the formula below that computes the intersection numbers at 
the relevant points of $d=0$:

\begin{proposition}\label{3.3}
  Consider $n \geq 3$ and $m \geq n$. Let $1 \le p < q \le n$ with $\gcd(p,q)=1$ and $q \neq 2$. 
If $c^{-1}_{p,q} = - {4 \cos^{2} (\pi p/q)}$, then
$$\overline{\cX_n} \bullet_{\mathbf{c}_{p,q}} \overline{\cY_m} = \nu_q(n) \nu_q(m).$$
where $\mathbf{c}_{p,q}=[c_{p,q}:0:1]$.
\end{proposition}

We have seen that the points at infinity on $X_n$ and $Y_m$ are described by the multipliers of their fixed points. So the multiplier of a fixed point is a natural parameter to use near infinity in $X_n$ (or $Y_n$).  To prove the above proposition, it is therefore convenient to work with the {\it totally marked moduli
  space $\mtwofm$}, which is a smooth complex manifold of dimension two (see~\cite[Lemma 6.6.]{M2}), namely, the space of quadratic rational maps with
marked critical points and marked fixed points. The elements of $\mtwofm$ are conjugacy classes
of sextuples $(f,\omega_1, \omega_2, x_1, x_2, x_3)$, where $f$ is a quadratic rational map,
$\omega_j$ are critical points of $f$, and $x_1, x_2, x_3$ is a complete list of the fixed points
of $f$ (with repetitions when $f$ has a multiple fixed point). As before, two such sextuples
are conjugate if there is a M\"obius transformation conjugating the rational maps and 
respecting markings. The conjugacy class of $(f,\omega_1, \omega_2, x_1, x_2, x_3)$ will be denoted
by $[f,\omega_1, \omega_2, x_1, x_2, x_3]$.

In $\mtwofm$ we let $A_n$ (resp. $B_m$) be the curves where the first (resp. the second) critical point is periodic of period $n$ (resp. $m$).
Following  7.4 of \cite{R2}, to study the ends of $A_n$ and $B_m$ it is convenient to consider the family of 
quadratic rational maps defined for $(\zeta,\rho) \in \C^* \times \C \setminus \{-1,-2\}$:
 $$h_{\zeta ,\rho }(z)=\zeta z\left( 1-{2+\rho \over
2(1+\rho )}z\right) \left( 1-{2\over
2+\rho }z\right) ^{-1}$$
$$=\zeta z\left(  1-{\rho ^{2}z\over 4(1+\rho )(1+{\textstyle {1\over
2}}\rho -z)}\right) .$$
A similar
parametrisation was used by James Stimson in his thesis \cite{Sti}. The great advantage of this type of parametrisation is that $h_{\zeta ,0}^k(z)$ is just $\zeta ^kz$, and estimates on $h_{\zeta ,\rho }^k(z)-\zeta ^kz$ for small $\rho $, and for a suitable set of $z$, are easily obtained.
  
The
critical points of $h_{\zeta ,\rho }$ are $\omega_1=1$ and $\omega_2=1+\rho $.  There
are two distinguished fixed points, one at $x_1=0$ and the other at $x_2=\infty $.  The third fixed point $x_3$ is given by the formula
$$x_3=\dfrac{(1+\frac{1}{2}\rho )(1+\rho )(\zeta -1)}{\zeta \rho ^2+(1+\rho )(\zeta -1)}.$$
Therefore, we may parametrise a subset $\cS$ of $\mtwofm$  by $\C^* \times \C \setminus \{0,-1,-2\}$,  using $[h_{\zeta ,\rho }, 1, 1 + \rho, 0, \infty, x_3]$. 

A partial compactification of $\cS$ is achieved by adding the line 
$\mathbb V(\rho )=\{ \rho=0\}$.
That is, we identify $\cS$  with a subset of $\CPtwo $, by identifying $(\zeta ,\rho )$ with $[\zeta :\rho: 1]$.
Now let $\cA_n = A_n \cap \cS$ and $\cB_m = B_m \cap \cS$.
Consider the closures $\overline{\cA_n}$ and $\overline{\cB_m}$ in $\CPtwo $.
The following was proved in Stimson's Thesis.

\begin{theorem}[(Stimson~\cite{Sti})]\label{3.4}
  The following statements hold:
  \begin{enumerate}
  \item The intersection of $\overline{\cA_n}$ (resp. $\overline{\cB_n}$) 
with the line $\mathbb V(\rho )$ is contained in
$$\{ [1:0:1]\}  \cup \{ [\exp(2 \pi i p/q):0:1] \mid 1 \le p < q \le n, \gcd (p,q)=1 \}.$$
From now on, write $\kappa =\exp (2\pi ip/q)$.
  \item Let   $E$ (resp. $F$) 
be an irreducible germ of the curve $\overline{\cA_n}$ (resp. $\overline{\cB_n}$) at $[\kappa :0:1]$.
Then, 
$$\dfrac{\zeta - \kappa}{\kappa \rho} \to a_1 \in \{ a \in \C \mid \Re a > 0 \},$$
as $E \ni (\zeta, \rho) \to (\kappa, 0)$ and,
$$\dfrac{\zeta - \kappa}{\kappa \rho} \to a_2 \in \{ a \in \C \mid \Re a < 0 \},$$
as $F \ni (\zeta, \rho) \to (\kappa, 0)$.
  \item 
 The intersection number of $\overline{\cA_n}$ (resp. $\overline{\cB_n}$) with the line $\mathbb V(\rho )$ at 
$[\kappa :0:1]$ is equal to the number of hyperbolic components of period $n$
 in the $p/q$-limb
of the Mandelbrot set.

  \end{enumerate}
\end{theorem}

\begin{proof}
From the formula we see that 
\begin{equation}\label{3.4.1}h_{\zeta ,\rho }(1)=\zeta \left( 1-\frac{\rho }{2(1+\rho )}\right).\end{equation}
For $\rho $ near $0$,  and for $z\ne 1+\frac{1}{2}\rho +O(\rho ^2)$, 
\begin{equation}\label{3.4.2}h_{\zeta ,\rho }(z)=\zeta z(1+o(1)).\end{equation}
Therefore, for $\rho $ near $0$, unless $\zeta ^q$ is close to $1$ for some $1\le q\le k$, we have $h_{\zeta ,\rho }^\ell (1)=\zeta ^\ell z(1+o(1))$, which is bounded from $1$ for $0<\ell\le k$. Putting $k=n$, we see that we can only have $(\zeta ,\rho )\in \cA_n$ for $\rho $ close to $0$ if $\zeta ^q$ is close to $1$ for some $1\le q\le n$. So (i) follows.

 For part (ii) write:
 \begin{equation}\label{3.4.3}h_{\rho }(1+z\rho )=\zeta (1+z\rho )\left( 1-\frac{\rho }{4(1+\rho )(\frac{1}{2}-z) }\right) .\end{equation}
  It follows from (\ref{3.4.1}) and (\ref{3.4.2}), and using a local Puiseux series in $\rho $ for $\zeta -\kappa$, that $\zeta -\kappa =O(\rho )$ in a neighbourhood of $[\kappa :0:1]$ in $\overline{\cA_n}$, and hence, if $\rho \to 0$ then $(\zeta - \kappa)/(\kappa \rho) \to a$ for some $a\in \C$. So,  for sufficiently small $\rho $, given $z$ in a compact subset of $\C\setminus \{1/2\} $, by the definition of $h_{\zeta ,\rho }$ and (\ref{3.4.2}) and (\ref{3.4.3}), applying $h_{\zeta ,\rho }$ to $h_{\zeta ,\rho }^{j-1}(1+\rho z)$ for $2\le j\le q$, 
  $$h_{\zeta ,\rho }^q(1+z\rho )=\zeta ^q(1+z\rho )\left( 1+\frac{\rho }{4(z-\frac{1}{2})}+o(\rho )\right) =1+\rho \left( qa+z +\frac{1}{4(z-\frac{1}{2})}\right) +o(\rho ).$$
 So, uniformly in compact subsets of $\C \setminus \{1/2\}$,
$$\dfrac{h^q ( 1 + z \rho) -1}{\rho} \to qa + z +\dfrac{1}{4(z-\frac{1}{2})}= g_a(z).$$
The map $g_a$ has a parabolic fixed point at $\infty$ and critical points at $z=0$ and $z=1$.
If $\Re a \leq 0$, it is not difficult (by a direct calculation) to check that  $\Re g^k_a(0)$ is a strictly decreasing sequence and $g^k_a(0)$  diverges to $\infty$.
Similarly, if $\Re a \geq 0$, then $\Re g^k_a(1)$ is a strictly increasing sequence and $g^k_a(1)$ diverges to  $\infty$. Part (ii) of the  theorem follows for $\cA_n$.  The result for $\cB_n$ follows immediately because, from the conjugacy by $z\mapsto (1+\rho )z^{-1}$, we see that $(\zeta ,\rho )\in \cA _n$ if and only if $(\zeta _1,\rho )\in \cB_n$, where 
$$\zeta _1=\zeta ^{-1}\frac{4(1+\rho)}{(2+\rho )^2}=\zeta ^{-1}(1+O(\rho ^2)).$$

\medskip
We now consider (iii), so let  $\kappa =1$ or $\kappa =\exp (2\pi ip/q)$ where $q>1$ and $1\le p<q$, and $p$ and $q$ are relatively prime. The multipliers at the fixed points $0$ and $\infty $ are $\zeta $ and $\zeta _1$ respectively. By (ii), either of these multipliers is a local coordinate  on $\cA_n$ in a deleted neighbourhood of $\mathbb V(\rho )\cap \overline{\cA_n}$. So, considering the multiplier $\zeta _1$, the intersection number of $\overline{\cA_n}$ with $\mathbb V(\rho )$ at $[\kappa :0:1]$ is the number of points $[\zeta :\rho :1]\in \cA_n$, for any fixed $\zeta _1$ sufficiently near $\kappa ^{-1}$, but not equal to it. 

If $\kappa =1$ then this number is at least $2^{n-1}$.  We can see this,
because after blowing up the point $[\zeta =1: \rho=0:1]$ we may observe that
in the coordinates
given by $(\rho, a = (\zeta -1)/\rho)$ the intersection of the proper transform of $\overline{\cA_n}$ with  the line $\rho =0$ is the set of parameters $a$ 
such that $g^n_a (0) =0$, which has degree $2^{n-1}$ in $a$.
 
  If $\kappa =\exp(2\pi ip/q)$, then the number is at least the number of hyperbolic components in the $-p/q$ limb of period $n$. In fact, consider a parameter $v$ in the $-p/q$ limb of the Mandelbrot set such that
the critical point of $Q_v(z) =z^2 + v$ is periodic of period dividing $n$.
Then, for any $0 < r < 1$ and $\zeta_1 (r) = r \exp (2 \pi i p/q)$ there exist 
$\rho(r)$ and $\zeta(r)$ such that the map $h_{\rho(r), \zeta(r)} \in \cA_n$ has an attracting fixed point of multiplier $\zeta_1(r)$ at $\infty$ and in the complement of
the basin of $\infty$ is hybrid equivalent to $Q_v$. According to Petersen~\cite[Corollary~2]{Pe},
as $r \nearrow 1$, we have that $\zeta_1 (r) \zeta (r) \to 1$, thus $\rho(r) \to 0$.
Hence, the intersection number of $\overline{\cA_n}$ with $\V(\rho)$ at 
$[\kappa =\exp (2 \pi i p/q): 0 : 1]$ is at least the number of hyperbolic components of period $n$ in the $-p/q$ limb of the Mandelbrot set, which via $z \mapsto \bar{z}$ is easily seen to coincide with the corresponding number in the 
$p/q$ limb.

  To prove equality, we need to compute the degree in $\zeta $ of the curve $\cA_n$ in $\CPtwo $. We can then obtain the result for $\cB_n$ by the usual birational equivalence. In fact it suffices to compute the degree in $\zeta $ of $H_{n,1}^h-H_{n,2}^h$ where
$$\frac{H_{n,1}}{H_{n,2}}=h_{\zeta ,\rho }^n(1)$$
and $H_{n,j}^h$ is the homogenised version of $H_{n,j}$. The zero set of $H_{n,1}^h-H_{n,2}^h$ is the union of sets $\cA_d$ for $d$ dividing $n$. We claim that the degree is $2^n-1$. Since the total number of hyperbolic components of periods dividing $n$ is $2^{n-1}$ by~\cite{D-H1}, and all but one of these are in limbs of the Mandelbrot set, the exception being the main cardioid, this suffices to show that for each $d$ dividing $n$, $\cA_d$ occurs with multiplicity one in the zero set of $H_{n,1}^h-H_{n,2}^h$, and that we have the equality claimed in (iii).

We write $\deg _\zeta (H)$ for the degree of a polynomial $H$ in $\zeta $. It is easily checked that $H_{k,1}$ has $\zeta $ as a factor for all $k\ge 1$, and hence has no constant term, and that $H_{k,2}$ has a constant term for all $k\ge 1$. The recursive equations are
$$H_{k+1,1}=\zeta H_{k,1}(2(1+\rho )(2+\rho )H_{k,2}-(2+\rho )^2H_{k,1})$$
$$H_{k+1,2}=2(1+\rho )H_{k,2}((2+\rho )H_{k,2}-2H_{k,1})$$
It follows that $H_{k,1}$ has $\zeta ^{2^k-1}$ as a factor, has  a term $(-4)^{2^{k-1}}\zeta ^{2^k-1}$  and has no constant term, and that $H_{k,2}$ has constant term $2\cdot 4^{k-1}$. It also follows that $\deg _\zeta (H_{k,2})\le 2^k-k-1$. 
Thus, the degree of $H_{n,1}^h-H_{n,2}^h$ is $2^n -1$ and the Theorem follows.

 \end{proof}

 \begin{uremark}
   Stimson's result is in fact far more precise than the theorem we have stated and proved above. 
Via the Thurston class of obstructed postcritically finite branched coverings,
he gives a complete characterisation of hyperbolic components
in the $-p/q$ limb that lead to the same irreducible germ of $\overline{\cA_n}$ at a point $[\exp(2 \pi i p/q):0:1]$. However, his proof that there is only one
such hyperbolic component per irreducible germ is not correct.
 \end{uremark}

\begin{lemma}\label{3.5}
  Let $1 \le p < q \le \min\{n,m\}$ with $\gcd(p,q)=1$ and $q \neq 2$. 
Let $c_{p,q}^{-1} = - {4 \cos^{2} (\pi p/q)}$ and $\kappa _{p,q}=\exp(2 \pi i p/q)$. Write ${\bf c}_{p,q}=[c_{p,q}:0:1]$ and ${\bf k }_{p,q}=[\kappa _{p,q}:0:1]$. Then
$$\overline{\cV_n} \bullet _{{\bf c}_{p,q}} \overline{\cW_m} = 
\overline{\cA_n} \bullet_{{\bf k }_{p,q}} \overline{\cB_m}.$$
\end{lemma}


\begin{proof} It is sufficient to show that there is a biholomorphic map between
a neighbourhood $U$ of ${\bf c}_{p,q}$ in the $(c,d)$-plane, and a neighbourhood $U'$ of 
${\bf  k }_{p,q}$ in the $(\zeta,\rho)$-plane, that maps $\cV_n$ to $\cA_n$ and
$\cW_m$ to $\cB_m$. If $U$ is small, then for any  $(c,d)\in U$, the multipliers of two of  the fixed points of $f_{c,d}$ are close to $\exp (\pm 2\pi ip/q)$, and if $d\ne 0$ the multiplier of the third fixed point is large. By taking $U$ sufficiently small, we can ensure that the multipliers of any map in $U$ are all different from the multipliers the fixed points of any other map in $U$, and that none of the multipliers is $1$. Therefore we know that none of the maps $f_{c,d}$, for  $(c,d)\in U$, is M\"obius conjugate to any other, and the fixed points vary holomorphically in $U$. We write $x_1(c,d)$ and $x_2(c,d)$ for the fixed points which have multipliers close to $\exp (2\pi ip/q)$ and $\exp (-2\pi ip/q)$ respectively. Since $x_1(c_{p,q},0)$ and $x_2(c_{p,q},0)$ are the solutions of $z^2-z+c_{p,q}=0$, the points $x_1(c,d)$, $x_2(c,d)$ and $0$ are all distinct for all $(c,d)\in U$, if $U$ is sufficiently small. Let $\tau _{c,d}$ denote the unique  M\"obius transformation which maps  $x_1(c,d)$, $x_2(c,d)$ and $0$ to $0$, $\infty $ and $1$. Then $\tau _{c,d}\circ f_{c,d}\circ \tau _{c,d}^{-1}$ is a quadratic rational map with fixed points $0$ and $\infty $ and critical point at $1$. It is therefore of the form $h_{\zeta ,\rho}$ with $\zeta =\zeta (c,d)$ and $\rho =\rho (c,d)$. In fact we have
$$\zeta(c,d)=f_{c,d}'(x_1(c,d)),$$
and
$$\rho (c,d)=\tau _{c,d}(-2d/c)-1.$$
The map $(\zeta ,\rho )$ is clearly a holomorphic injection on $U$ which maps ${\bf c}_{p,q}$ to ${\bf k}_{p,q}$, because all the maps in $U$ are distinct up to M\"obius conjugacy. Hence it must be a holomorphic bijection onto a neighbourhood $U'$ of ${\bf k}_{p,q}$. In any case, the inverse map is easily obtained by using conjugation by a M\"obius transformation which maps $1$, $h_{\zeta ,\rho }(1)$ and $h_{\zeta ,\rho }^2(1)$ to $0$, $\infty $ and $1$. Note that $q=2$ has to be omitted, because we need the points $1$, $h_{\zeta ,\rho }(1)$ and $h_{\zeta ,\rho }^2(1)$ to be bounded apart.
\end{proof}



The proof of Proposition~\ref{3.3} follows, since (ii) of Theorem~\ref{3.4} implies that at an
intersection point of $\overline{\cA_n}$ and  $\overline{\cB_m}$  in $\rho =0$, the tangents  to $\overline{\cA_n}$ and the tangents to $\overline{\cB_m}$ 
are distinct.

\subsection{Low periods and the proof of Theorem~\ref{1.1}}\label{4}

The last ingredients needed to prove  Theorem~\ref{1.1} are contained in the lemma below. 

\begin{lemma}
\label{lem:9}  
 $$\begin{array}{l}\eta_{\rm IV}(1,m) = 2^{m-1},\\ \\
 \eta _{\rm IV}(1,m,p/q)=\nu _q(m),\\ \\
\eta_{\rm II} (2) = 1,\\ \\
\eta_{\rm IV}(2,m)-\eta_{\rm IV}(1,m)=\dfrac{1}{3} 2^m - \dfrac{(-1)^m}{3}.\end{array}$$
Here, $\eta _{\rm IV}(1,m,p/q)$ denotes the number of type IV hyperbolic components in the $p/q$-limb of the Mandelbrot set  with an attractive cycle of period dividing $m$.
\end{lemma}
\begin{proof} 
As already mentioned, the formula for $\eta _{\rm IV}(1,m)$ is a consequence of the Douady-Hubbard classification \cite{D-H1} of hyperbolic components in the Mandelbrot set for the parameter family of quadratic polynomials $\{ z^2+v\mid v\in \C\} $, which is naturally identified with $\C$. In fact, the complement of the Mandelbrot set in $\C $ is the intersection with $\C$ of the type I hyperbolic component of rational maps which is mentioned in the introduction. The type I hyperbolic component contains no critically finite maps. In fact, it coincides with the set of $v$ for which the forward orbit of $0$ under $z\mapsto z^2+v$ diverges to $\infty $.  The main cardioid of the Mandelbrot set (as it is known) is the  hyperbolic component of $z\mapsto z^2$. The set of all the other  hyperbolic components in the Mandelbrot set, for which the attractive cycle  is of period dividing $m$, is  in two-to-one correspondence with the set of odd-denominator rationals 
$$\left \{ \frac{k}{2^m-1}:0<k<2^m-1\right \} .$$
Therefore, we obtain
$$\eta_{\rm IV}(1,m) =1+\frac{2^{m}-2}{2}=2^{m-1}.$$
Furthermore, the set of all hyperbolic components in the $p/q$ limb of the Mandelbrot set, for which the attractive cycle is of period dividing $m$, is in two-to-one correspondence with the set of odd-denominator rationals $k/(2^m-1)$ in a closed interval $[1/(2^q-1),2/(2^q-1)]$, since the arguments of the root are of the form $r/(2^q-1)$ and $(r+1)/(2^q-1)$ for some $r$ coprime to $q$. This is the number $\nu _q(m)$. 

The fact that $\eta _{\rm II}(2)=1$ follows immediately from fixing critical points of a quadratic rational map to be $0$ and $\infty $. If these are in a period two cycle then the map must be of the form $z\mapsto \lambda z^{-2}$ for some $\lambda \neq 0$, and any  such maps are conjugate to $z\mapsto z^{-2}$, by a conjugacy of the form $z\mapsto \mu z$. 

We now give a proof of the formula for $\eta_{\rm IV}(2,m)-\eta_{\rm IV}(1,m)$ which is close to the methods of Section~\ref{3.2}. Since a great deal is known about $V_2$, other proofs are possible. For example, one can use the results of a recent paper of Aspenberg-Yampolsky, \cite{Asp-Yam}, some of which occur also in work of Timorin \cite{Tim}, and both of which consolidate the structure that has been known in outline for some time, as evidenced, for example, by the thesis of Luo \cite{Luo} in the 1990's. One can also make use of the theory of matings, which involves using Thurston's criterion~\cite{D-H2} and 
Tan Lei's theorem to decide which matings are realisable~\cite{TL}.

Instead of using $\cR$ as in the case of $V_n$ for $n\geq 3$, we use a parametrisation which extends that used in \cite{Asp-Yam} and \cite{Tim}:
$$g_{a,b}(z)=\frac{a}{z^{2}+2z+b}.$$
If $a\neq 0$, critical points of $g_{a,b}$ are $\infty $ and $-1$, and $g_{a,b}(\infty )=0$. Every M\"obius conjugacy class in $\mtwocm \setminus X_1$, apart from that of $z\mapsto z^{-2}$, is represented by exactly one map $g_{a,b}$ for $a\neq 0$, and therefore $\mtwocm \setminus (X_1\cup\{  z^{-2}\} )$ can be identified with 
$$\{ [a:b:1]\mid a,b\in \C,a\neq 0\} \subset \CPtwo  .$$
 We write $\cZ_m$ for the set of $[a:b:1]$ such that $g_{a,b}$ represents an element of  $Y_{m}$. Then 
 $$\eta _{\rm IV}(2,m)-\eta _{\rm IV}(1,m)=\#( \cZ _m\cap \{ [a:0:1]\mid a\neq 0\} ).$$
 As in the earlier cases, all intersections are transverse. We write  $\overline{\cZ _m}$ for the closure of $\cZ _m$ in $\CPtwo $. The plane $b=0$ in $\CPtwo $ is the set  of all $[a:0:1]$ together with $[1:0:0]$. It will be denoted by $\V(b)$, the vanishing set of the linear polynomial $b$. Provided that all zeros are simple, by Bezout's Theorem,
$$\# (\overline{\cZ _m}\cap \V(b))=\deg (\overline{\cZ _m}). $$

Therefore, 
$$\eta _{\rm IV}(2,m)-\eta _{\rm IV}(1,m)=\deg (\overline{\cZ _m})-\overline{\cZ_{m}}\bullet  _{[0:0:1]}\V (b)-\overline{\cZ_{m}}\bullet _{[1:0:0]}\V(b)$$
where $\overline{\cZ_{m}}\cdot _{[0:0:1]}\V(b)$ denotes the intersection number of $\overline{\cZ_{m}}$ with $\V(b)$ at $[0:0:1]$. To compute $\deg (\overline{\cZ _m})$, we write 
$$g_{a,b}^{m}(-1)=\frac{T_{m}(a,b)}{U_{m}(a,b)}$$
Then $\cZ_{m}$ is the zero set of $T_{m}^{h}+U_{m}^{h}$, where, following the notation of Section \ref{2.2}, $T_{m}^{h}$ and $U_{m}^{h}$ are the homogenised versions of the polynomials $T_{m}$ and $U_{m}$. Induction shows that the degrees of $T_{m}$, $U_{m}$  and $T_m+U_m$  are $2^{m}-1$ for all $m\geq 1$.  In fact, the highest degree terms of $T_{m+1}$ and $U_{m+1}$ come from the highest degree terms of $aU_m^2$ and $bU_m^2$ respectively. As in Section \ref{2.2} we see that, for all $m\geq 3$, there is a birational map $\psi :\CPtwo \to \CPtwo $ such that
$$P_m^h\circ \psi =\begin{cases} T_{m}^{h}+U_{m}^{h}&{\rm{\ if\ }}m{\rm{\ is\ odd}}\\ \\ \dfrac{T_{m}^{h}+U_{m}^{h}}{T_{2}^{h}+U_{2}^{h}}&{\rm{\ if\ }}m{\rm{\ is\ even.}} \end{cases}$$
Hence,  as in Section \ref{2.2}, since $P_m^h$ is square-free for all $m\geq 3$, and since $T_m^h+U_m^h$ is square-free for $m=1$ or $m=2$, this is also true for $T_m^h+U_m^h$ for all $m\geq 3$. So 
$$\deg (\cZ_m)=\deg (T_m+U_m)=2^m-1.$$
So now we need to compute $\overline{\cZ_{m}}\bullet  _{[0:0:1]}\V(b)$ and $\overline{\cZ_{m}}\bullet  _{[1:0:0]}\V(b)$. We use homogenised coordinates $[a:b:t]$. We write 
$$\begin{array}{ll} T_m(a,b)=T_m^1(a)+bT_m^2(a,b),\ \ &T_m^h(a,b,t)=T_m^{h,1}(a,t)+bT_m^{h,2}(a,b,t),\\ \\U_m(a,b)=U_m^1(a)+bU_m^2(a,b),\ \ &U_m^h(a,b,t)=U_m^{h,1}(a,t)+bU_m^{h,2}(a,b,t)\end{array}$$
Then $\overline{\cZ_{m}}\bullet  _{[0:0:1]}\V(b)$  and $\overline{\cZ_{m}}\bullet _{[1:0:0]} \V(b)$  are the maximum powers of $a$ and $t$ respectively which divide $T_m^{h,1}+U_m^{h,1}$. The first of these is the maximum power of $a$ dividing $T_m^1+U_m^1$. The second is $\deg (T_m+U_m)-\deg (T_m^1+U_m^1)$. So if we write
$$g_{a,0}^m(-1)=\frac {T_m^0(a)}{U_m^0(a)}$$
where $T_m^0$ and $U_m^0$ have no non-zero power of $a$ as a common factor, and, provided that $T_m^0+U_m^0$ has a non-zero constant term, we see that 
$$\eta _{\rm IV}(2,m)-\eta _{\rm IV}(1,m)=\deg (T_m^0+U_m^0).$$

A straightforward induction gives
$$T_1^0=a,\ \ U_1^0=-1$$
$$T_2^0=1,\ \ U_2^0=a-2,$$
$$\begin{array}{l}T_{2k+1}^0=a(U_{2k}^0)^{2},\\ \\
U_{2k+1}^0=T_{2k}^0(T_{2k}^0+2U_{2k}^0),\\ \\
T_{2k+2}^0=(U_{2k+1}^0)^{2}=(T_{2k}^0)^2(T_{2k}^0+2U_{2k}^0)^2,\\ \\
U_{2k+2}^0=T_{2k+1}^0(T_{2k+1}^0+2U_{2k+1}^0)/a=(U_{2k}^0)^2(a(U_{2k}^0)^2+2(T_{2k}^0)^2+4T_{2k}^0U_{2k}^0).\end{array}$$
Induction also gives $\deg (U_{2k}^0)>\deg (T_{2k}^0)$ for all $k\geq 1$ and $\deg (T_{2k+1}^0)>\deg (U_{2k+1}^0)$ for all $k\geq 0$. Therefore the degree of $T_m^0+U_m^0$ is the maximum of the degrees of $T_m^0$ and $U_m^0$. More precisely, we have
$$\begin{array}{l}\deg (T_{2k+1}^0)=2\deg (U_{2k}^0)+1,\\ \\ \deg (U_{2k+1}^0)=\deg (T_{2k}^0)+\deg (U_{2k}^0),\\ \\  \deg (T_{2k+2}^0)=2\deg (U_{2k+1}^0)=2(\deg (T_{2k}^0)+\deg (U_{2k}^0)),\\ \\ \deg (U_{2k+2}^0)=4\deg (U_{2k}^0)+1.\end{array}$$
This gives
$$\begin{array}{l}\deg (U_{2k}^0)=\sum _{i=0}^{k-1}4^i=\dfrac{4^k-1}{3}=\dfrac{2^{2k}-1}{3},\\ \\ \deg (T_{2k+1}^0)=\dfrac{2^{2k+1}+1}{3}.\end{array}$$
So we obtain
$$\deg (T_m^0+U_m^0)=\frac{2^{m}-(-1)^{m}}{3}.$$
It remains to show that the constant term of $T_m^0+U_m^0$ is non-zero. For this, it suffices to show that if $a_{2k}$ and $b_{2k}$ are the constant terms of $T_{2k}^0$ and $U_{2k}^0$, then 
$$b_{2k}<0<a_{2k}<-b_{2k}$$
for all $k\ge 1$. This is true for $k=1$. It then follows for all $k$ by induction, from
$$a_{2k+2}=a_{2k}^2(a_{2k}+2b_{2k})^2,\ \ \ b_{2k+2}=a_{2k}b_{2k}(2b_{2k})(a_{2k}+2b_{2k}).$$

\end{proof}

\section{Type II components}\label{5}

As before, we let $\eta _{\rm II}(m)$ denote the number of type II components of period dividing $m$, in the space $\cM _2^{cm}$, and $\eta _{\rm II}'(m)$ is the number of type II components of period exactly $m$. For $1\leq j<m$, let $\eta _{\rm II}(m,j)$ denote the number of type II components of period $\geq 3$ dividing $m$ such that the second marked critical point $\omega _2$ is in the $j$'th iterate of the immediate attractive basin of $\omega _1$. Note that
$$\sum _{d\mid m,d\geq 3}\dfrac{m}{d}\eta _{\rm II}'(d)=\sum _{j=1}^{m-1}\eta _{\rm II}(m,j).$$ 
 We have already seen that $\eta _{\rm II}(2)=\eta _{\rm II}(2,1)=1$. The aim of this section is to prove Theorem~\ref{1.2}. In fact, we obtain a sharper result by giving a formula for $\eta _{\rm II}(m,j)$. 

\begin{theorem}\label{5.1} For all $m\geq 3$, and all $1\leq j<m$,
\begin{equation}\label{5.1.1}\begin{array}{ll}\eta _{\rm II}(m,j)=&\dfrac{7}{36}2^m-\dfrac{1}{12}(2^j+2^{m-j})-\dfrac{(-1)^m}{36}((-2)^j+(-2)^{m-j})\\ \\
\ &-\dfrac{1}{4}-\dfrac{5}{36}(-1)^m+\dfrac{1}{12}((-1)^j+(-1)^{m-j})\\ \\
\ &-\dfrac{1}{2}{\displaystyle \sum _{3\leq q}\phi (q)\nu_{q}(j)\nu_{q}(m-j).}\end{array}\end{equation}

Hence 
\begin{equation}
 \begin{array}{ll}{\displaystyle \sum _{d\mid m,d\geq 3}\dfrac{m}{d}\eta'_{\rm II}(d)}=&
\dfrac{7}{36}m2^m-\dfrac{37}{108}2^m-\dfrac{m}{4}-(-1)^m\dfrac{5}{36}m+\dfrac{1}{2}+(-1)^m\dfrac{5}{54}\\ \\
\ &-\dfrac{1}{2}{\displaystyle \sum _{3\le q \le j\le m-q}\phi (q)\nu_q(j)\nu_q(m-j).}\end{array}\end{equation}
\end{theorem}


\begin{uremark} \noindent 1. The formula is symmetric in $j$ and $m-j$, as expected, because we can interchange $\omega _1$ and $\omega _2$.

\noindent 2. In particular,
$$\begin{array}{l}
\eta _{\rm II}({3,1})=\eta _{\rm II}({3,2})=1,\\ \\
\eta _{\rm II}({4,1})=\eta _{\rm II}({4,2})=\eta _{\rm II}({4,3})=2,\\ \\
\eta _{\rm II}({5,j})=5{\rm{\ \ for\ \ }}1\leq j\leq 4,\\ \\
\eta _{\rm II}({6,j})= 10{\rm{\ \  for \ \ }}1\le j\le 5,\\ \\
\eta _{\rm II}({7,j})=\begin{cases}21{\rm{\ \ for\ \ }}j=1,2,5,6,\\ 22{\rm{\ \ for\ \ }}j=3,4,\end{cases}\\ \\
\eta _{\rm II}({8,j})=\begin{cases}42{\rm{\ \ for\ \ }}j=1,2,3,5,6,7,\\ 44{\rm{\ \ for\ \ }}j=4.
\end{cases}\end{array}$$
The only  non-zero contribution to $\eta _{\rm II}({m,j})$ from the last row of (\ref{5.1.1}), in this list, is when $m=6$, $j=q=3$, or $m=7$, $j=3$ or $4$ and $q=3$, or $m=8$ and $j=3$ or $5$ and $q=3$, or $j=4$ and $q=3$ or $4$.
\end{uremark}

\subsection{Outline Proof}
\label{5.2} 

We use the representation of the elements of $\cR=\cM_2^{cm}\setminus X_2$ by maps $f_{c,d}$, where $(c,d) \in \C \times \C^*$ as in Section~\ref{2.1}.
Given $m \geq 3$, for $j \geq 1$ we consider the curves in $\cR$:
$$\cP_j = \left \{(c,d) \in \cR \mid  f_{c,d}^j(0)=-\frac{2d}{c}\right  \}$$
and 
$$\cQ_{m-j} =\left  \{(c,d) \in \cR \mid  f_{c,d}^{m-j} \left(-\dfrac{2d}{c} \right)=0 \right  \}.$$
The number $\eta _{\rm II}({m,j})$ is then given by the cardinality of
$\cP_j \cap \cQ_{m-j}$.
All intersections between $\cP_j$ and $\cQ_{m-j}$  are transverse, by Corollary~\ref{cor:TransversePreperiodic}.

As  in the case of type IV components, we consider the closure of $\cP_j$ and $\cQ_{m-j}$ in
$\CPtwo$, and apply Bezout's Theorem.

We will start by computing the degree of  $\overline{\cP_j}$ and $\overline{\cQ_{m-j}}$ and then continue to subtract the intersections at ``infinity'' from the product of the degrees.
The extra difficulty for counting intersections at infinity arises 
from the fact that we will have to establish an analogue  Stimson's Theorem~\ref{3.4}. 
Both computing the degrees, and computing the intersections at infinity, 
will rely heavily on the parametrisation of the unique Type I component 
introduced in~\cite{R4}.

\subsection{The $v$ coordinate}\label{5.7} 
We proceed to summarise the relevant results contained in~ \cite{R4} related to the parametrisation of the Type I hyperbolic component.

For each $\zeta \in \D^* = \{ \zeta \mid 0 < |\zeta | <1 \}$ we consider the quadratic rational map
$$\tau_\zeta (z) = z \dfrac{z + \zeta}{1 + \bar{\zeta} z}.$$
In terms of the notation of Section~\ref{type-param}, $$\beta _{-\zeta }=\frac{1+\bar{\zeta }}{1+\zeta }\tau _{\zeta }.$$
Note that both $z =0$ and $z=\infty$ are  attracting fixed points of $\tau_\zeta$. The basin of $z=0$ is $\D$ and contains a unique critical point
which we denote by $c(\zeta)$. 
The dynamics in $\D$ is semiconjugate, via the K\"onigs coordinate, to multiplication by $\zeta$.
The  K\"onigs coordinate  is the unique holomorphic map $\phi_\zeta : \D \to \C$ such that $\zeta \phi_\zeta (z) = \phi_\zeta \circ \tau_\zeta (z)$ 
and $\phi_\zeta (0)=0, \phi_\zeta (c(\zeta)) =1$. Such a map is an isomorphism between a neighbourhood of the origin and the unit disk.
More precisely, there exists a conformal map $\psi_\zeta : \D \to \psi_\zeta(\D)$ such that 
$\psi_\zeta(0) =0$ which is an inverse of $\phi_\zeta$ (that is, $\phi_\zeta \circ \psi_\zeta (z) = z$,
and for all $w \in  \psi_\zeta(\D)$, we also have that $\psi_\zeta \circ \phi_\zeta (w) = w$).

\medskip
Now consider an element $[f,\omega_1,\omega_2] \in \cM^{cm}_2$ which lies in the Type I component. Then $f$ has an attracting fixed point $z_0$, which we assume has multiplier $\zeta \neq 0$. Similarly to above, there exists a conformal isomorphism
$\psi_f : \D \rightarrow \psi_f(\D)$ with image contained in the basin of $z_0$ such that $\psi_f(0) = z_0$, $\psi_f(\zeta z) =f(\psi_f(z))$ and at least one critical point of $f$ lies in $\partial \psi_f(\D)$. 
(This map $\psi_f$ always extends to a homeomorphism from $\overline{\D}$ onto its image.) 

Let $\cU$ be the subset of the Type I component formed by all $[f,\omega_1,\omega_2]$
with an attracting fixed point of non-vanishing multiplier  such that
$\omega_1 \notin \partial \psi_f(\D) \ni \omega_2$. Note that $\cU$ is well-defined, since
the required properties are invariant under conjugacies that respect critical markings.

Given $[f,\omega_1,\omega_2] \in \cU$, with attracting fixed point $z_0$ with multiplier $\zeta=\zeta(f)$, and basin of attraction $U_f$, there exists a unique  
K\"onigs coordinate $\phi_f:U \to \C$ that semiconjugates $f$ with multiplication
by $\zeta(f)$ such that $\phi_f(\omega_2)=1$, and which is equal to $\psi _f^{-1}$ near $z_0$. In the sequel, we will always assume that
$\phi_f$ is this unique semiconjugacy.
In a neighbourhood of $z_0$, the map $\psi_\zeta \circ \phi_f$ is a conjugacy between $f$ and $\tau_\zeta$. 
Taking iterated preimages, the conjugacy $\psi_\zeta \circ \phi_f$ uniquely extends 
to a simply connected domain contained in $U_f$ containing the critical value $f(\omega_1)$.
The $v$-coordinate $v(f)$ of $f \in \cU$ is,  by definition, the image of $f(\omega _1)$ under the conjugacy extending $\psi_\zeta \circ \phi_f $.

According to~\cite{R4},
$$\begin{array}{ccc}
\cU & \to & \D^* \times \D \\
\left[ f , \omega_1, \omega_2 \right] & \mapsto & (\zeta(f), v(f))
\end{array}
$$
is a real-analytic homeomorphism.

\subsection{The degrees of the curves}

\begin{lemma}
  \label{lem:1}
 Let $m \geq 3$. The following statements hold:
\begin{itemize}
\item For all $j \ge 1$, the degree of $\overline{\cP_j}$ is
$$ \dfrac{1}{6}(2^{j}-(-1)^{j})+\dfrac{1}{2}.$$
\item For all $j \ge 1$,  the degree of $\overline{\cQ_{m-j}}$ is
$$\dfrac{1}{6}(7 \cdot 2^{m-j}- (-1)^{m-j})-\dfrac{1}{2}.$$
\end{itemize}
\end{lemma}

\begin{proof}
  We write $(cP_j+2dQ_j)^{h}$ for the homogenised version of $cP_j+2dQ_j$. Similarly,
$R_{m-j}^{h}$ is the homogenised version of $R_{m-j}$ where $P_j,Q_j,R_{m-j}$ are
as in Section~\ref{2.2}. Observe that $\overline{\cP_j}$ and $\overline{\cQ_{m-j}}$ are the varieties of $(cP_j+2dQ_j)^{h}$ and $R_{m-j}^{h}$, respectively.
 Using parts (i) and (iv) of Lemma~\ref{lem:4} in Section~\ref{2.2} we obtain the first line below, and part (vi) of the same lemma implies the second:
\begin{eqnarray*}
  \deg ((cP_j+2dQ_j)^{h}) &=&\dfrac{1}{6}(2^{j}-(-1)^{j})+\dfrac{1}{2},\\
  \deg (R_{m-j}^{h}) &=& \dfrac{7}{6}2^{m-j}-\dfrac{1}{6}(-1)^{m-j}-\dfrac{1}{2}.
\end{eqnarray*}
  Hence to establish the lemma we need to show that these polynomials are square free.
From the birational equivalence argument we just need to check that $cP_j+2dQ_j$ is square-free. 

Given $0<r<1$, let $L_r$ be 
the set formed by the maps $f_{c,d}$ in
$\cR$ that have an attracting fixed point with multiplier
$r$. From~\cite{R4},  the intersection of $\cP_j$ with $L_r$
is contained in $\cU$ and, as summarised above, it is in bijection
with the  points $z \in \D$ which map onto $c(r)$
under $j-1$ iterates of the map $\tau_r$. It follows that the cardinality of the intersection
$\cP_j \cap L_r$ is $2^{j-1}$, all of them transverse.

It is sufficient to show that the number of intersections of $\overline{\cP_j}$ and
$\overline{L_r}$ coincides with the product of the degrees of $cP_j+2dQ_j$ and
the polynomial equation of $L_r$, which we proceed to compute.

The curve $L_r$ is given by the equations
$$1+\dfrac{c}{z}+\dfrac{d}{z^2}=z,\ \ -\dfrac{c}{z^2}-2\dfrac{d}{z^3}=r.$$
These equations give 
$$z=\dfrac{1\pm \sqrt{1+c(2+r)}}{2+r}$$
Substituting for $z$, the curve $\zeta =r$ becomes the degree three curve in $c$ and $d$:
$$\dfrac{c^2r^2}{(2+r)^4}(8-c(2+r))-\dfrac{c}{(2+r)^4}(c(2+r)^2+4r)^2$$
$$+\dfrac{cr}{(2+r)^5}(c(2+r)^2+4r)(4-2c(2+r))+4d^2+\dfrac{4d}{(2+r)^3}(c(2+r)^2+4r)$$
$$+\dfrac{12dcr}{(2+r)^2}=0$$

The only term of degree three is $c^3$ and so there are no elements of $\overline{L_r}$ 
at $[c:0:1]$ for $c\neq 0$.  Since both the above equation of $L_r$ and
$cP_j+2dQ_j$ have linear terms but no constant terms (Lemma~\ref{lem:4} (iii)),  
$[0:0:1] \in \overline{\cP_j} \cap \overline{L_r}$ and the intersection number is $1$. 

Finally,  $[0:1:0]$
lies in the intersection only when $j$ is odd.
By Lemma~\ref{lem:4} (ii) and (iv), when $j$ is even, both $P_j$ and $Q_j$ have the same degree and both have powers of
$d$ of maximal degree, thus $[0:1:0] \notin \overline{\cP_j}$ if $j$ is even.
When $j$ is odd, $\deg(P_j) = \deg(Q_j) +1$, so $[0:1:0] \in \overline{\cP_j}$
and the intersection number at $[0:1:0]$ is also $1$.

Thus, the total intersection number 
between $\overline{\cP_j}$ and $\overline{L_r}$ is
$2^{j-1} + 1 + (1-(-1)^j)/2$ which is equal to $3 \deg (cP_j +2dQ_j)$.
Hence, $cP_j +2dQ_j$ is square free.
\end{proof}

\subsection{No intersections at $c=d=\infty$ and intersections at $c=0$.}

\begin{lemma}\label{5.3} We have the following.
\begin{itemize} 
\item $[1:s:0] \notin \overline{\cP_j} \cap \overline{\cQ_{m-j}}$ for all $s \in \C$.
\item $\overline{\cP_j}\bullet_{[0:0:1]} \V(d)=1$, where $\V(d)$ denotes the plane $d=0$,  and $\overline{\cP_j}\bullet_{[0:0:1]}\overline{\cQ_{m-j}}=2^{m-j-1}$ for all $1\leq j<m$,
\item $\overline{\cP_j}\bullet_{[0:1:0]}\overline{\cQ_{m-j}}=\dfrac{1}{6}(1-(-1)^{j})(2^{m-j}-(-1)^{m-j})$.
\end{itemize}
\end{lemma}

\begin{proof} 
\noindent{\em{Intersections at $[1:s:0]$.}} 
For $s \neq 0$ this follows directly from Lemma~\ref{lem:7}, since $\overline{\cP_j} \cap \overline{\cQ_{m-j}} \subset \overline{\cX_m} \cap \overline{\cY_{m}}$. 
There are no intersections at $[1:0:0]$ either, because $cP_j+2dQ_j$  always has a term $c^k$ of maximal degree by Lemma~\ref{lem:4} (iv).

\medskip
\noindent {\em{Intersections at $[0:0:1]$.}}
By (v) of Lemma~\ref{lem:4},  the minimum degree $\mindeg (R_{m-j})$ of the monomials  in $R_{m-j}(c,d)$ is $2^{m-j-1}$ and  this is uniquely realised by a monomial in $d$ for all 
$1\leq j<m$.
By (iii) of Lemma~\ref{lem:4}, the constant term of both $P_j$ and $Q_j$ is $1$  for all $j\geq 3$. Hence $cP_j+2dQ_j$ has non-zero linear terms in $c$ and $d$ and the claim follows for all $j\geq 3$. Since $\cP_1 = \{c=0 \}$ and $\cP_2 = \{ c+2d =0\}$, the claim also follows for $j=1,2$.

\medskip 
\noindent{\em{Intersections at $[0:1:0]$.}}
There are no intersections if $j$ is even, because, by (ii) of Lemma~\ref{lem:4}, both $P_j$ and $Q_j$ have the same degree, and, by (iv) of Lemma~\ref{lem:4}  both have powers of $d$ of maximal degree. So the degree of $cP_{j}+2dQ_j$ is realised by a power of $d$ in $2dQ_j$. But $\deg(P_j)=\deg(Q_j)+1$ for all  odd $j\geq 3$ by (ii) of Lemma~\ref{lem:4}, and hence the degree of intersection of $(cP_j+2dQ_j)^h=0$ with $c=0$ is $1$ for all odd $j\geq 3$ (using, again, that linear terms in $c$ and $d$ are non-zero). The degree of the intersection of $(cP_j+2dQ_j)^h$ with $R_{m-j}^h$ is then given by $\deg (R_{m-j})-\deg (R_{m-j}^1)$, where we write
$$R_k(c,d)=R_k^1(d)+cR_k^2(c,d),\ \ S_k(c,d)=S_k^1(d)+cS_k^2(c,d).$$
Inductively we see that
$$R_1^1=S_1^1=4d,$$
$$R_{k+1}^1=(R_k^1)^2+d(S_k^1)^2,\ \ S_{k+1}^1=(R_k^1)^2,$$
and hence for all $k\geq 1$,
$$\deg (R_{2k+1}^1)=2\deg (R_{2k}^1)=1+4\deg (R_{2k-1}^1).$$
It follows that
$$\deg (R_{m-j}^{1})=\frac{5}{6}2^{m-j}-\frac{1}{2}+\frac{1}{6}(-1)^{m-j},$$
and hence
$$\deg (R_{m-j})-\deg (R_{m-j}^1)=\frac{1}{3}(2^{m-j}-(-1)^{m-j}).$$
\end{proof}

\subsection{Preliminaries on intersections at $[c:0:1]$.}\label{5.5}

\begin{lemma}\label{5.4}
If $[c:0:1]\in \ocP_j \cup \ocQ_j$, then  $c=0$, or $c=c_{p,q}$ for some $q$ with   $3\le q\le j$ and   $p$ with $\gcd(p,q)=1$, where  $c_{p,q}$ as in Proposition~\ref{3.3}, that is, $c_{p,q}^{-1}=-4\cos ^2(\pi p/q)$.
\end{lemma}

\begin{proof}
If $f_{c,0}\in \overline{\cP_j}$ for $c\neq 0$, then it is the limit of maps $f_{c_{n},d_{n}}$ in $\cP_j$ for a sequence $(c_n,d_n)$. The critical points of $f_{c_{n},d_{n}}$ are $0$ and $-2d_n/c_n$, where $-2d_n/c_n\to 0$. The critical values are $\infty $ and $1-c_n/4d_n^2$, which converges to $\infty $, and hence $f_{c_n,d_n}^2(0)$ converges to $1$.   Since $f_{c_n,d_n}^j(0)=-2d_n/c_n$ and since  the maps $f_{c_{n},d_{n}}$ converge uniformly to $f_{c,0}$ outside any neighbourhood of $0$, restricting to a subsequence  there must be some least integer $3\leq k \leq j$ such that $f_{c_n,d_n}^k(0)\to 0$, and hence $f_{c,0}^k(0)=0$ and $f_{c,0}$ has order $k$.   So $c=c_{p,q}$ as claimed. The proof for $\ocQ_j$ is similar.
\end{proof}

 In order to compute intersection numbers,  
we use the family of maps $h_{\zeta ,\rho }$ introduced in Proposition~\ref{3.3}, and varieties $\cC_j$ and $\cD_{m-j}$ corresponding to M\"obius conjugates in the $h_{\zeta ,\rho}$ family of the maps $f_{c,d}$ in $\cP_j$ and $\cQ_{m-j}$. 
That is, $\cC_j$ consists of those parameters $(\zeta,\rho)$ such that
the critical point $\omega_1=1$ of $h_{\zeta,\rho}$ maps in $j$ iterates onto 
the critical point $\omega_2=1+\rho$. Similarly, $\cD_{m-j}$ consists of  those parameters $(\zeta,\rho)$ such that
the critical point $\omega_2=1+\rho$ of $h_{\zeta,\rho}$ maps in $j$ iterates onto 
the critical point $\omega_1=1$. 
As in Theorem~\ref{3.4} we may restrict our attention to study the intersections of $\ocC_j$ and $\ocD_{m-j}$ with the line $\rho=0$.
In particular, we will only be interested on parameters $(\zeta, \rho)$ where
$|\rho|$ is close to $0$.

 \begin{lemma}\label{5.6} 
 Let $\ocC_j$ and $\ocD _{m-j}$ be the varieties in the $(\zeta ,\rho )$ coordinates that correspond to $\ocP_j$ and $\ocQ_{m-j}$. Then, for $c_{p,q}^{-1}=-4\cos ^2(\pi p/q)$ and  $\kappa _{p,q}=e^{2\pi ip/q}$, we have the following.

\begin{enumerate}
\item
The following numbers coincide, where ${\bf c}_{p,q} =[c_{p,q}:0:1]$ and ${\bf k}_{p,q} =[\kappa _{p,q}:0:1]$:

$\ocP_j\bullet _{{\bf c}_{p,q}} \V(d)$,

$\ocQ_j\bullet  _{{\bf c}_{p,q}} \V(d)$, 

$\ocC_j\bullet _{{\bf k}_{p,q}} \V(\rho )$,

$\ocD_j\bullet _{{\bf k}_{p,q} } \V(\rho )$.

\item
$$\ocP_j\bullet _{{\bf c}_{p,q}}\ocQ_{m-j} \\ =(\ocP_j\bullet _{{\bf c}_{p,q}} \V(d)) \cdot  (\ocQ_{m-j} \bullet_{{\bf c}_{p,q}} \V(d)).
$$

\end{enumerate}
\end{lemma}

\begin{proof} 
For (i), note that the biholomorphism of the proof of Lemma~\ref{3.5} establishes
that the first and third number coincide, as well as the second and fourth.
Now $(\zeta, \rho) \mapsto (\zeta, -\rho/(1+\rho))$ interchanges $\cC_j$ and 
$\cD_j$. Hence, their intersection numbers at ${\bf k}_{p,q}$ coincide.

(ii) Once we have established (i),
this is proved very similarly to the absence of common tangent lines
at $\V(\rho )$ of  $\cA_m$ and $\cB_n$ in Theorem~\ref{3.4}. As there, we have 
$$\dfrac{h^q ( 1 + z \rho) -1}{\rho} \to a + z +\dfrac{1}{2(2z-1)}= g_a(z).$$
If $\Re a \leq 0$, then $\Re g^k_a(0)$ is a strictly decreasing sequence and  if $\Re a \geq 0$, then $\Re g^k_a(1)$ is a strictly increasing sequence. Therefore we must have $\zeta =\kappa _{p,q}(1+a_1 \rho +o(\rho )) $ with $\Re (a_1)>0$ on any branch of $\cC_j$ near $(\kappa ,0)$ and $\zeta =\kappa _{p,q}(1+a_2\rho +o(\rho ))$ with $\Re (a_2)<0$ on any branch of $\cD_{m-j}$ near $(\kappa _{p,q},0)$. The absence of common tangent lines follows, and (ii) is a consequence of this and the biholomorphism of the proof of Lemma~\ref{3.5}.
\end{proof}

\subsection{Intersection number at ${\bf k}_{p,q}$}
\begin{theorem}
  \label{thr:1}
Let $\kappa _{p,q}= \exp (2 \pi i p/q)$ and ${\bf k}_{p,q}=[\kappa _{p,q}:0:1]$ (as before).
Then  $$\ocC_j\bullet _{{\bf k }_{p,q}} \V (\rho )= \nu_q(j).$$
\end{theorem}

To compute $\cC_j\bullet _{{\bf k }_{p,q}} \V(\rho)$, we observe that there exists $\delta >0$, such that this number coincides 
with the cardinality of
$$\cC_j \cap  (\{ r\kappa \} \times \{ \rho\mid |\rho| \leq \delta \} )$$
for all $0<r<1$ sufficiently close to $1$. 
According to~\cite{R4}, the parameters in the above intersection correspond to maps in the subset $\cU$
of the Type I component introduced in  Section~\ref{5.7}.
To prove the theorem, we need 
to further understand the image of the above intersection under
the parametrisation described in Section~\ref{5.7}.
The description of this image is contained in the next lemma and proposition.

Recall that $\phi_\zeta : \D \to \C$ denotes the (normalised) K\"onigs coordinate
for $\tau_\zeta$.

\begin{lemma}
  \label{lem:2}
Consider the graph 
$$S_\kappa 
= \{ t \kappa^i \mid 1 \le i \le q, t \geq 0 \} \subset \C$$
Given $0 < r < 1$, let $\zeta = r \kappa$ and consider
$${\Gamma}= {\Gamma}_\zeta = \phi_\zeta^{-1} (S_\kappa) \subset \D.$$

Then the following statements hold:

\begin{itemize}
\item $\Gamma$ is connected, simply connected, and locally homeomorphic to a finite tree.
\item ${\Gamma} \setminus \{ 0 \}$ has exactly $q$ connected components.
\item Label by $\gamma^\zeta_1$ the connected component of ${\Gamma} \setminus \{ 0 \}$ containing $\tau_\zeta (c(\zeta))$.
Then $$\# \{ z \in \gamma^\zeta_1 \mid \tau^{n-1}_\zeta (z) = c(\zeta) \} = \nu_q(n).$$


\end{itemize}
\end{lemma}

\begin{proof}
  We can write $\Gamma $ as an increasing union of sets $\Gamma _n$ for $n\geq 0$, where  
$$\Gamma _n=\tau _{\zeta }^{-n}(\psi_\zeta (S_\kappa \cap \D )).$$  
For each $n$, $\tau _\zeta :\Gamma _{n+1}\to \Gamma _n$ is a degree
two branched cover with a single critical point. So, by induction on
$n$, each $\Gamma _n$ is a finite tree.  The only intersections
between $\overline{\Gamma _{n+1}\setminus \Gamma _n}$ and $\Gamma _n$
are at extreme points of $\Gamma _n$, again by induction.  It follows
that $\Gamma $ is a connected and locally finite tree. Also by
induction, $\Gamma _n\setminus \{ 0\} $ has $q$ components, and each
one of the $q$ components of $\Gamma _{n+1}\setminus \{ 0\} $ contains
one of the $q$ components of $\Gamma _n\setminus \{ 0\} $, which, in
turn, contains one of the points $\tau _\zeta ^j(0)$ for $1\leq j\leq
q$ . So taking the union of all of these, $\Gamma \setminus \{ 0\} $
also has $q$ components $\gamma _j^\zeta $, for $1\leq j\leq q$, where
$\gamma _j^\zeta $ contains $\tau _{\zeta}^j(c(\zeta ))$ --- and
$\gamma _q^\zeta $ also contains $c(\zeta )$.  Also, $\tau _{\zeta }$
maps $\gamma _j^\zeta $ homeomorphically onto $\gamma _{j+1}^\zeta $
if $1\leq j\leq q-1$, and maps $\gamma _q^\zeta $ onto $\Gamma $,
mapping one-to-one onto $(\Gamma \setminus \gamma _1^\zeta )\cup \{
\tau _\zeta (c(\zeta))\} $ and two-to-one onto $\gamma _1^\zeta
\setminus \{ \tau _\zeta (c(\zeta ))\} $. It follows that the number
$a_j(n)=\# (\tau _\zeta ^{-n}(c(\zeta))\cap \gamma^\zeta _j)$ satisfies
$$a_j(0)=\begin{cases}0{\rm{\ if\ }}j<q,\\
  1{\rm{\ if\ }}j=q,\end{cases}$$
$$a_j(n+1)=\begin{cases}a_{j+1}(n){\rm{\ if\ }}j<q,\\
2a_1(n)+\sum _{1<i\leq q}a_i(n){\rm{\ if\ }}j=q.\end{cases}$$

Let $s(n) = a_1(n) + \cdots +a_q(n)$. It follows that
$s(n+1) = 2 s(n)$. Thus $s(n) = 2^n$.
Moreover, given $k \ge 0$ and $0 \le r <q$, since $a_1(kq+r) = a_q((k-1)q+r+1)=a_1((k-1)q+r) + s((k-1)q+r)$,
we obtain that 
$$a_1 (kq+r) = a_1(r) + s(r) + \cdots +  s((k-1)q+r) = a_{1+r}(0)+s(r+q) + \cdots +  s((k-1)q+r)$$
$$=\begin{cases}\dfrac{2^{kq+r}-2^r}{2^q-1}&{\rm{\ if\ }}r<q-1,\\ 1+\dfrac{2^{kq+r}-2^r}{2^q-1}&{\rm{\ if\ }}r=q-1.\end{cases}$$
Setting $n= kq+r +1$ we obtain that $a_1(n-1) =\nu_q(n)$.

\end{proof}

\begin{proposition}
  \label{pro:1}
  Let $X$ be a branch of $\ocC_j$ at $(\kappa,0)$.
There 
exist $0< r_0 <1$ and $\delta >0$ such that
if $h = h_{r \kappa, \rho} \in X$ for some $r_0 < r < 1$ and $|\rho| < \delta$, then
$v(h) \in \gamma^{r\kappa}_1$. 
\end{proposition}

We assume this proposition and defer, for a moment, its lengthy proof.  
 
\begin{proof}[of Theorem~\ref{thr:1}]
From the previous proposition and lemma we obtain that
$$\ocC_j \bullet_{{\bf k}_{p,q}}\V(\rho )\le \nu_q(j).$$

Applying Bezout's Theorem we obtain the first line below.
From Lemma~\ref{5.4} we obtain the second identity. The third line is obtained putting together the inequality above with Lemma~\ref{5.6} and Lemma~\ref{5.3}. 
The fourth line follows since the number of hyperbolic components of period
dividing $j$ in the $p/q$ limb is $\nu_q(j)$ (Lemma~\ref{lem:9}).
Then the fifth line is deduced using the fact that $\eta_{IV}(1,j) = 2^{j-1}$ of
Lemma~\ref{lem:9}. After rearranging the terms in the sixth line, the last identity is obtained from Lemma~\ref{lem:1}.

\begin{eqnarray*}
\deg \ocP_j & = &  \ocP_j \bullet \V(d)\\ 
& = & \ocP_j \bullet_{[0:0:1]}\V(d)+ \sum_{{\bf c}_{p,q}, 3 \le q \le j} \ocP_j \bullet_{{\bf c}_{p,q}} \V(d) \\
& \le & 1 + \dfrac{1}{2}\sum_{3 \le q \le  j} \phi(q) \nu_q(j) = 1+ \dfrac{1}{2}\sum_{2 \le q \le j} \phi(q) \nu_q(j) -
\dfrac{\nu_2(j)}{2}\\
& = & 1+ \dfrac12( \eta_{IV}(1,j) -1) - \dfrac12 \left( \dfrac{2^{j-1} + (-1)^j}3 \right)\\ 
& = & \dfrac12( 2^{j-1} - \dfrac{2^{j-1} + (-1)^j}3) + \dfrac12\\
& = &  \dfrac{1}{6}(2^{j}-(-1)^j) + \dfrac12 \\
& = & \deg \ocP_j.
\end{eqnarray*}
Thus, equality holds throughout. In particular, by Lemma~\ref{5.6},
$$\nu_q(j)= \ocP_j\bullet _{{\bf c}_{p,q}} \V (d) = \ocC_j \bullet_{{\bf k}_{p,q}}\V(\rho ).$$
\end{proof}

\subsubsection{Proof of Proposition~\ref{pro:1}}
The proof of the proposition relies on the two lemmas
below which loosely speaking say that
the K\"onigs coordinate $\phi_h$ and its inverse $\psi_h$ 
are close  to the identity in $\D$, 
as $r \to 1$, since $h$ converges to the identity in $\D$.

\begin{lemma}
  \label{lem:3}
  Let $X$ be a branch of $\ocC_j$ at $(\kappa,0)$.
Given $\varepsilon >0$, there 
exists $0< r_0 <1$, $C >0$ and $N \in \N$ such that
if $h = h_{r \kappa, \rho} \in X$ 
and $r_0 < r <1$,  then
$$|\phi_h (z) - z | < \varepsilon |z|$$
for all $|z| \leq 1- C(1-r)$ and,
$$|h^{N} (1+\rho) | \leq 1 - C(1-r).$$
\end{lemma}

\begin{proof}
Note that, there exists $a \in \C$ with
$\Re {a} > 0$  such that, if  $h_{\zeta, \rho} \in X$, then
$$\zeta = \kappa (1 + a \rho + o(\rho)),$$
as $\rho \to 0$. 
Assume that  $h_{r \kappa, \rho} \in X$ and $r <1$. Let $\Delta =1-r$. 
It follows that
$$\rho = - \dfrac{\Delta}{a} (1+o(\Delta))$$
as $\Delta \to 0$.

Given $\varepsilon >0$, take 
$$C \geq \dfrac{1}{2 \varepsilon |a|^2} + \dfrac{1}{|a|}.$$
Assume that
$$|z| \leq 1- C \Delta.$$
Then, for $\Delta$ sufficiently small,
$$|1 + \dfrac{\rho}{2} -z | \geq  -\dfrac{\Delta}{2|a|} (1+o(\Delta)) + C \Delta
\geq  \dfrac{\Delta}{2 \varepsilon |a|^2}.$$
Also recall that $h_{\zeta, \rho} (z) = \zeta z P_\rho(z)$ where
$$P_\rho(z) = 1 -\dfrac{\rho^2 z}{4(1+\rho)(1+\rho/2-z)}.$$
Thus, for $\Delta$ sufficiently small, 
$$|P_\rho(z) -1| \leq \dfrac{\left(2\Delta/|a|\right)^2}{2 \Delta/ (2 \varepsilon
 |a|^2)}    \leq 4 \varepsilon \Delta         |z|.$$
Hence, 
$$|h (z)| \leq (1 -\Delta) |z| (1+ |P_\rho(z) -1|)  \leq   (1 -\Delta) |z| (1 +4 \varepsilon \Delta ) \leq  (1 - (1-4 \varepsilon ) \Delta) |z|.$$
In particular, $|h(z)| \leq 1- C \Delta.$

Let $z_n = h^n(z)$ and recall that $\phi_h(z) = \lim \zeta^{-n} z_n$.

\begin{eqnarray*}
  \left|\dfrac{z_{n+1}}{\zeta^{n+1}} - \dfrac{z_n}{\zeta^n} \right| & = & (1-\Delta)^{-n} |z_n| \cdot |P_\rho(z_n) - 1| \\
 & \leq & (1-\Delta)^{-n} |z_n|^2 4 \varepsilon \Delta \\
 & \leq & (1-\Delta)^{-n} (1 - (1-4 \varepsilon ) \Delta)^{2n} |z_0| 4 \varepsilon \Delta.
\end{eqnarray*}
Hence,
$$|\zeta^{-(n+1)} z_{n+1} - z_0| \leq 4 \varepsilon \Delta \sum \left(\dfrac{(1 - (1-4 \varepsilon ) \Delta)^2}{1-\Delta} \right)^n |z_0|
= \dfrac{4 \varepsilon (1-\Delta)}{1 - 8 \varepsilon - (1- 4 \varepsilon)^2 \Delta}|z_0| \leq 8 \varepsilon|z_0|$$
for $\Delta$ sufficiently small and $\varepsilon$ small, say $\varepsilon <1/20$.
This proves the first assertion.

\smallskip
For the second assertion, recall that as $\rho \to 0$,
$$\dfrac{h^q(1 +z \rho) -1}{\rho} \to g(z) = qa + z + \dfrac{1}{2(2z-1)}$$
uniformly in compact subsets of $\C \setminus \{ 1/2 \}$.
Since $\infty$ is a parabolic fixed point of $g$ with attracting direction 
$(0, +\infty) \cdot a$, it follows that, given $\delta >0$, for $N$ large, $|g^N(1)|$ is large and 
$$|\arg g^N(1) - \arg a| < \delta.$$
 Since  
$\rho = -\Delta (1+ o(1))/a$,
we have that $$|1 + g^N(1) \rho| \leq 1 - (\cos \delta) |g^N(1)| \Delta.$$
Thus, taking $\delta >0$ small, $N$ sufficiently large, and $r$ sufficiently close to $1$,
we have that $(h^{qN}(1+\rho) -1)/\rho$ is sufficiently close to a sufficiently
large $ g^N(1)$ so that 
 $$|h^{qN}(1+\rho)| \leq 1 - C \Delta.$$
\end{proof}

\begin{lemma}
Let $X$ be a branch of $\ocC_j$ at $(\kappa,0)$.
Given $\varepsilon >0$, there 
exists $0< r_0 <1$ and $N' \in \N$ such that the following holds.
For all $\zeta = r \kappa$ with $r_0 < r <1$,  
and all  $h = h_{r \kappa, \rho} \in X$ 
if $$w \in \{t \zeta^n  \mid t \in ]0,1] \},$$
then
$$ |\psi_h(w) - w| < \varepsilon |w|,$$
for all $n \ge N'$.
\end{lemma}

\begin{proof}
Let $\Delta =1-r$.
Consider $N$ and $C$ as in the previous lemma.
Denote by $\operatorname{dist_\D}$ the hyperbolic distance (with constant curvature $-1$) in the open unit disk. 
Observe that
$$\operatorname{dist_\D} (\zeta^n, \zeta^{n+q}) = \log \dfrac{1+r^n}{1-r^n} \dfrac{1-r^{n+q}}{1+r^{n+q}} \leq q\Delta + \log 2.$$

Hence we may choose $N' > N$  such that
$$\operatorname{dist_\D} (\zeta^n, \zeta^{n+q}) < 1$$
for all $n \geq N'$ and all $\Delta$ sufficiently small.

We may assume that $\Delta$ is sufficiently small
so that the $\varepsilon/2$ ball centred at $\kappa^n$, denoted
$B_{\varepsilon/2}(\kappa^n)$ is not contained in $\psi_h(\D)$. In fact,
since $h$ converges uniformly in compact subsets of $\C\setminus \{1\}$
to multiplication by $\kappa$, any such ball contains an iterated preimage 
of the critical point $\omega_2=1$. 

We now claim that for all $n \ge N'$,
 if $\psi_h(\zeta^n)$ and $\psi_h(\zeta^{n+q})$ 
lie in $B_{\varepsilon/2}(\kappa^n)$,
then the hyperbolic geodesic of $\psi_h(\D)$ joining these points
is contained in  $B_{\varepsilon}(\kappa^n)$.
In fact, recall that the (infinitesimal) hyperbolic arc length in $\psi_h(\D)$
is bounded 
below by $1/(2 \delta(z))$ where $\delta(z)$ is the distance from
$z$ to the boundary of $\psi_h(\D)$. 
Thus the geodesic from  $\psi_h(\zeta^n)$ to $\psi_h(\zeta^{n+q})$ 
may not exit $B_{\varepsilon}(\kappa^n)$, for otherwise, it would have length
at least $1/2$. 

By the previous lemma we may take $\Delta$ sufficiently small
so that the following two conditions are satisfied:

$\bullet$ For all $w$ such that $|w| \le 1 -\varepsilon/5$,
$$|\psi_h(w) -w | < \epsilon |w|.$$

$\bullet$ For all $z$ such that $|z| < 1 -C\Delta$ we have
$$|\phi_h (z) -z| < \varepsilon/8.$$

$\bullet$ For all $n$ such that  $N' \le n < N'+q$,
we may also assume that $|\zeta^n- \kappa^n| < \varepsilon/8$.

Recall that $\psi_h(\zeta^k) = h^k(\omega_1)$, for all $k \geq 1$. Equivalently,
$\phi_h(h^k(\omega_1)) = \zeta^k$. 
We may assume that $n$ is such that $N' \le n < N'+q$, it follows that,
$$|h^n(\omega_1) - \kappa^n | \leq |h^n(\omega_1) - \zeta^n| +|\zeta^n - \kappa^n | < \varepsilon/2.$$

Let $m \geq 1$ be such that 
$$|h^{n+qm}(\omega_1) - \kappa^n | \leq \varepsilon/2$$
and
$$|h^{n+q(m+1)}(\omega_1) - \kappa^n | > \varepsilon/2.$$
Note that this number $m$ depends on $\Delta$ (that is on $\zeta$).

It follows that
$$|\zeta^{n+q(m+1)} - \kappa^n | >  \varepsilon/2 - \varepsilon/8.$$
Taking $\Delta$ smaller if necessary, so that the second inequality
below holds we have:
$$|\zeta^{n+q(m+1)}-\zeta^{n+qm}| \leq |\zeta^q - 1| < \varepsilon/8.$$
Hence,
$$|\zeta^{n+qm} - \kappa^n | >  \varepsilon/2 - \varepsilon/8 - \varepsilon/8>\varepsilon/5.$$

If $w = s\zeta^{n}$, for some $s < r^{n+qm}$
 then $|w| \leq 1-\varepsilon/5$ and
therefore $|\psi_h(w) -w| \le \varepsilon |w|$.

Finally, if $w = s \zeta^n$ where  $r^{n+qj} \leq s \leq r^{n+q(j-1)}$
for some $1 \le j \le m$, then $\psi_h(w)$ lies in a geodesic joining
the points $h^{n+q(j-1)}(\omega_1)$ and $h^{n+qj}(\omega_1)$, which are
$\varepsilon/2$-close to $\kappa^n$. Thus, $\psi_h(w)$ is $\epsilon$-close 
to $\kappa^n$. For $\Delta$ sufficiently small, it follows that
$|\psi_h(w) -w| \le 3 \varepsilon |w|$.
\end{proof}

\begin{proof}[of Proposition~\ref{pro:1}]
We write $\gamma _i'$ for the image of $\gamma _i^\zeta $ under the conjugacy between $\tau _\zeta $ and $h$ extending $\phi _h^{-1}\circ \psi _\zeta $. We also write $\Gamma _h'$ for the union of  the $\gamma _i'$. It  is sufficient to show that $h(\omega_1) \in \gamma_1'$ for all
$r$ sufficiently close to $1$. 

Again we let $\Delta =1-r$ and let $\varepsilon >0$ be sufficiently
small so that the sectors 
$$S_i = \{ z \in \C \setminus \{0\} | |\arg z - \arg \kappa^i | < 3 \varepsilon \}$$
for $i=0,\dots, q-1$ (subscripts $\mod\ q$) are pairwise disjoint.

For $\Delta$ sufficiently small and $N'$ as in the previous lemma with $q|N'$, 
for all $i=0,\dots,q-1$, we have that
 $\psi_h(]0,1]\zeta^{N'+i}) \subset S_i$.
We will show that $\gamma'_i \subset S_i$. Since $h(\omega_1) \in S_1$
the proposition will follow.

Let $T_0 = \cup _{i\ge 0}\psi_h([0,1]\zeta^{N'+i})$ 
and note that $h^{M}(\Gamma'_h) \subset T_0$ for $M=N'+q+j$.

We claim that 
$$\Gamma'_h \subset S=\cup _{i=0}^{q-1}S_i,$$
from which it follows immediately that $\gamma _i'\subset S_i$ for all $0\le i\le q-1$. In order to prove this statement,
given $\delta>0$, let us denote by $V_\delta$ a $\delta$-open neighbourhood
of the roots of unity $1,\kappa,\dots,\kappa^{q-1}$. 
Since $h^M$ converges uniformly (spherical metric) 
to multiplication by $\kappa^M$ in $\bar{\C} \setminus V_\varepsilon$ as $\Delta \to 0$, we have that
for $\Delta$ sufficiently small, if $z \in   \bar{\C} \setminus V_\varepsilon$
and $h^M(z) \in T_0 \cap  V_{2 \varepsilon}$, then $z \in T_0$. 
We may also assume that if $z \in   \bar{\C} \setminus V_\varepsilon$ and
$h^M(z) \in V_{2 \varepsilon}$, then $z \in V_{3 \varepsilon}$.
Thus, given $z \in \Gamma'_h$, then $h^M(z) \in T_0$ and therefore,
$z \in V_{3 \varepsilon} \cup T_0\subset S$, as required.
\end{proof}

\subsection{Proof of Theorem~\ref{5.1}}
By Bezout's Theorem, and our study of intersections at infinity,
 we have that $\eta_{\rm II}(m,j)$ is obtained
by subtracting from the product of the degrees
of $\ocP_j$ and $\ocQ_{m-j}$ (first line below) the intersection numbers at
$[0:0:1]$ (second line), at $[0:1:0]$ (third line), and at all ${\bf c}_{p,q}$
(fourth line).  
\begin{equation*}
\begin{split}
\eta_{\rm II}(m,j) =&  \left(\dfrac{1}{6}(2^{j}-(-1)^{j})+\dfrac{1}{2} \right)
\left(\dfrac{1}{6}(7 \cdot 2^{m-j}- (-1)^{m-j})-\dfrac{1}{2}\right) \\
              -&  2^{m-j-1}\\
              -&  \dfrac{1}{6}(1-(-1)^{j})(2^{m-j}-(-1)^{m-j})\\
              -& \dfrac12 \sum_{3 \le q \leq j} \phi(q) \nu_q(j)\nu_q(m-j).
\end{split}
\end{equation*}   
Now a calculation shows that the formula above is equivalent to that stated
in the theorem. \hfill $\Box$

\newpage
\noindent
{Jan Kiwi \\ Facultad de Matem\'aticas, \\  Pontificia Universidad Cat\'olica de Chile.\\
\email{www.mat.puc.cl/\~{}jkiwi/}}

\bigskip
\noindent
{Mary Rees \\ Department of Mathematics, \\  University of Liverpool, \\ United Kingdom.\\
\email{www.liv.ac.uk/\~{}maryrees/maryrees.homepage.html/}}
\end{document}